\documentclass[11pt,oneside]{amsart}
\usepackage{amscd,amssymb}
\usepackage{color}
 \usepackage{graphicx}
 
 \usepackage{todonotes}

\usepackage{graphicx, amsmath, amssymb, amscd, amsthm, euscript, psfrag, amsfonts,bm,hyperref}

\DeclareMathAlphabet{\mathpzc}{OT1}{pzc}{m}{it}

\newtheorem{theorem}{Theorem}[section]
\newtheorem{lemma}[theorem]{Lemma}
\newtheorem{proposition}[theorem]{Proposition}
\newtheorem{corollary}[theorem]{Corollary}

\newcommand{\BMS}{\operatorname{BMS}}

\newenvironment{definition}[1][Definition]{\begin{trivlist}
\item[\hskip \labelsep {\bfseries #1}]}{\end{trivlist}}

\newenvironment{remark}[1][Remark]{\begin{trivlist}
\item[\hskip \labelsep {\bfseries #1}]}{\end{trivlist}}

\usepackage{amssymb}

\DeclareMathAlphabet{\mathpzc}{OT1}{pzc}{m}{it}

\newtheorem{thm}[theorem]{Theorem}

\newtheorem{prop}[theorem]{Proposition}
\newtheorem{Nota}[theorem]{Notation}

\numberwithin{equation}{section}

\numberwithin{equation}{section}

\newcommand{\be}{begin{equation}}

\newcommand{\e}{{\epsilon}}
\newcommand{\z}{\mathbb{Z}}

\newcommand{\R}{\mathbb{R}}

\newcommand{{\grinv}}{{\Cal G}^{-r}}

\newcommand{\ba}{\backslash}

\newcommand{\G}{\Gamma}

\newcommand{\g}{\gamma}

\newcommand{\Haar}{\operatorname{Haar}}

\newcommand{\Cal}{\mathcal}

\newcommand{\la}{\langle}
\newcommand{\ra}{\rangle}

\newcommand{\SL}{\operatorname{SL}}

\newcommand{\bp}{\begin{pmatrix}}
\newcommand{\ep}{\end{pmatrix}}

\renewcommand{\bp}{{\rm bp}}

\newcommand{\SO}{\operatorname{SO}}

\newcommand{\PS}{\rm{PS}}

\newcommand{\supp}{\operatorname{supp}}

\newcommand{\BR}{\operatorname{BR}}
\newcommand{\Leb}{\operatorname{Leb}}

\renewcommand{\setminus}{-}

\renewcommand{\be}{\begin{equation}}
\newcommand{\ee}{\end{equation}}

\newcommand{\Ad}{\operatorname{Ad}}

\newcommand{\Z}{\z}
\newcommand{\B}{B}

\newcommand{\cyl}{\mathsf{C}}

\newcommand{\Pp}{ \sigma}
\newcommand{\poi}{\mathcal{P}}

\title[Mixing on convex cocompact manifolds]{Exponential mixing for frame flows for convex cocompact hyperbolic manifolds}

\date{}
\author{Dale WInter}

\begin{document} 
\maketitle

\begin{abstract} The aim of this paper is to establish exponential mixing of frame flow for the measure of maximal entropy on a convex cocompact hyperbolic manifold. Consequences include results on the decay of matrix coefficients and on effective equidistribution of holonomies. The main technical point is a spectral bound on certain ``twisted'' transfer operators, which we obtain by building on Dolgopyat's framework. This extends and strengthens earlier work of Dolgopyat, Stoyanov, Pollicott, and others. 
\end{abstract}

\section{introduction}
We write $G := \SO(n, 1)_+$ for the orientation preserving isometry group of real hyperbolic $n$-space $\mathbb{H}^n$. We seek to understand discrete subgroups $\G < G$ with their associated hyperbolic manifolds $X := \G \ba \mathbb{H}^n$. Standard questions in this setting include
\begin{enumerate}
\item Orbit counting: for a fixed base point $o\in\mathbb{H}^n$ and $T \in \mathbb{R}$, how many orbit points $\G o$ lie within distance $T$ of $o$?
\item Counting geodesics: how many closed geodesics are there on $X$ of length less than $T$?
\item Representation Theory: for the $G$-representation $L^2(\G \ba G)$ and unit vectors $v, w \in L^2(\G \ba G)$, can we say anything about the matrix coefficients $\langle v, g\cdot w \rangle$ as $g \rightarrow \infty$ in $G$?
\end{enumerate}
When $\G$ is a lattice in $G$ (that is when $X$ has finite hyperbolic volume) these questions are by now classical (see \cite{EM, Ma, HM}, for example). We are therefore interested in the case of non-lattices or, equivalently, infinite area hyperbolic surfaces. Unfortunately the class of all infinite area hyperbolic manifold class seems to be far too broad to be tractable. Instead we shall focus on on groups and manifolds that are ``convex-cocompact'', which is to say that the non-wandering set for their geodesic flows should be compact (see subsection \ref{convex_cocompact} for details).  The advantage of working with convex cocompact groups is that compactness of the non-wandering set allows certain tools from hyperbolic dynamics to be applied in a straightforward fashion. This observation underlies work of Lalley, Naud, Stoyanov, and others \cite{La, Na, St} on which we build. 

Our aim is to to prove three theorems, which we now describe. Suppose that $\G < G$ is a Zariski dense convex cocompact subgroup and write $K$ for the maximal compact subgroup of $G$ stabilizing some base point $o \in \mathbb{H}^n$, so that $G/K = \mathbb{H}^n$. Let $A = \lbrace a_t: t \in \mathbb{R} \rbrace $ denote a one parameter diagonalizable subgroup of $G$ and let $M = Z_K(A)$ be the centralizer of $A$ in $K$. It is well known that $\G \ba G$ and $\G \ba G / M$ correspond to the orthonormal frame bundle and the unit tangent bundle over $\G \ba \mathbb{H}^n$ respectively.  The right action of $A$ on $\G \ba G$ and $\G \ba G / M$ correspond to the frame flow and geodesic flow respectively. Reparametrizing  $A$ if necessary we assume that the geodesic flow so obtained proceeds at unit speed. 

The Haar measure on $\G \ba G$ can be difficult  to work with, being non-ergodic and totally dissipative whenever $\G$ is a convex-cocompact non-lattice. Instead we work with the (unique) measure of maximal entropy $m^{\BMS}$ for the $A$ action on $\G \ba G / M$. This approach grew out of work of Bowen, Margulis, and Sullivan, and is well described in Roblin's thesis \cite{Ro}. Abusing notation slightly we also write $m^{\BMS}$ for the $M$ invariant lift of $m^{\BMS}$ to $\G \ba G$. We write $\delta = \delta(\G) \in (0, n-1)$ for the critical exponent of $\G$ or, equivalently, for the topological entropy of $(\G \ba G/M, a_t)$. 

The dynamical formulation of our main result is as exponential mixing for $m^{\BMS}$.  
\begin{theorem} \label{main1} The $a_t$ action is exponentially mixing for $m^{\BMS}$: There exists $\eta > 0$ and $k \in \mathbb{N}$ such that for any compactly supported functions $\phi, \psi \in C^k(\G \ba G)$ that are $k$-times continuously differentiable we have
$$ \int_{G \ba G} \phi(ga_t) \psi(g) dm^{\BMS}(g) = m^{\BMS}(\phi)m^{\BMS}(\psi) + O(||\phi||_{C^k} ||\psi||_{C^k} e^{-\eta t}).$$
as $t \rightarrow \infty$: here $C^k$ denotes the usual $C^k$ norm with respect to the left invariant metric on G. 
\end{theorem}

\begin{remark}
The main term here already appears in work of Parry-Pollicott \cite{PP}. For functions invariant under the right $M$-action on $\G \ba G$ the error term is due to Stoyanov \cite{St} building on work of Dolgopyat \cite{Do}. Our approach is to extend their ideas to general (non $M$-invariant) functions. 
\end{remark}

Throughout this paper we will concentrate on smooth functions. In fact, though, a relatively simple convolutional argument, apparently going back to C.C. Moore and M. Ratner and generalized by Kleinbock and Margulis in \cite[Appendix]{KM}, extends the result to H\"older functions.
\begin{corollary} \label{main1_coro}
The estimate of Theorem \ref{main1} holds so long as $\phi, \psi$ are H\"older; in this case the $C^k$ norms in the error term should be replaced with appropriate H\"older norms, while the decay exponent $\eta$ will depend upon the H\"older class. 
 \end{corollary} 
 This corollary follows from Theorem \ref{main1} by choosing a sequence of approximations $\phi_\e, \psi_\e$ for $\phi, \psi$ as in, for example,  \cite[Lemma 2.4]{GS}, applying Theorem \ref{main1} to the correlation function of $\phi_\e, \psi_\e$, and then setting $\e$ to shrink exponentially in $t$, in the style of \eqref{choose_epsilon_carefully_for_t_trick}.

Using Roblin's transverse intersection argument, this can be rephrased as a statement on the decay of matrix coefficients for the Haar measure $dg$ on $\G \ba G$.
\begin{theorem} \label{main2} There exists $\eta > 0$ and $k \in \mathbb{N}$ such that for any compactly supported functions $\phi, \psi \in C^k(\G \ba G)$ we have an explicit constant $C(\phi, \psi)$ such that
$$ \int_{\Gamma \ba G} \phi(ga_t) \psi(g) dg =C(\phi, \psi) e^{( \delta - n + 1)t} + O(||\phi||_{C^k} ||\psi||_{C^k} e^{(\delta - n + 1 - \eta) t}).$$
as $t \rightarrow \infty$. The constant is given in terms of the Burger-Roblin measures of $\phi, \psi$. Note that the implicit constant here depends on the supports of $\phi, \psi$. 
\end{theorem} 
Stoyanov's work, together with the transverse intersection argument, is sufficient to prove this for $M$-invariant functions, so the novelty is that we do not restrict to trivial $M$-type. One can also extend to H\"older functions. Some care is needed here; if we try to pass directly from Theorem \ref{main2} to the corollary below there are difficulties. See subsection \ref{comment_on_holderness_for_haar} instead. 
\begin{corollary}\label{main2_coro}
The estimate of Theorem \ref{main2} extends to all H\"older functions; the smooth norms in the error term should then be replaced by H\"older norms, and the decay exponent depends on the H\"older class of $\phi, \psi$. 
\end{corollary}

These results are also known by representation theoretic methods for geometrically finite groups $\G$ with $\delta(\G) > (n-2) $ \cite{MO}. There are some challenges to extending their approach to smaller $\delta$, however. In the range $(n-1)/2 < \delta < n-2$, Mohammadi-Oh's argument provides these results conditional on the non-spherical spectral gap conjecture, \cite[Corollary 1.2]{MO}. In the case $0<\delta \leq (n-1)/2$ the representation theoretic path seems completely blocked by existence of continuous spectrum for the hyperbolic Laplacian, so the results here appear to be entirely new. Shifting to our approach, which is based in dynamical systems rather than representation theory, allows us to deal easily with very sparse groups with small critical exponent; the downside to  our approach is that it makes cusps more difficult to deal with. 

There is a well established framework relating mixing properties on the one hand with counting or equidistribution questions on the other \cite{Ma, EM, DRS, Ro, OS, KO, MO}. For our current setting of holonomies in $\G \ba G$ this theory was developed in work of Margulis-Mohammadi-Oh \cite{MMO}. Their arguments will allow us to rephrase our result as an equidistribution statement for holonomies. For any closed geodesic $\alpha$ on $\G \ba \mathbb{H}^n$ of length $\ell(\alpha)$ we have an associated holonomy class $h(\alpha)$, which is a conjugacy class in $M$. Denote by $\mathcal{G}(T)$ the collection of primitive closed geodesics in $\G \ba \mathbb{H}^n$ with length less than $T$.  Since the holonomy associated to a closed geodesic is only defined up to conjugation in $M$, the appropriate notion of equidistribution is with respect to class functions, that is, with respect to functions on $M$ that are constant on conjugacy classes. 

\begin{theorem} \label{main3}  There exists $\eta > 0$ such that for any smooth class function $\phi$ on $M$
$$ \sum_{\alpha \in \mathcal{G}(T)}\phi(  h(\alpha))  = \operatorname{li}(e^{\delta T})  \int_M \phi(m) dm + O(e^{(\delta - \eta)T}).$$
\end{theorem} 

\subsection{Sketch of proof} Our main aim is to prove Theorem \ref{main1}; from there Theorem \ref{main2} follows in a well established manner recounted in Section \ref{matrix_coefficient_section}. From Theorem \ref{main2} one then deduces Theorem \ref{main3} by the arguments of Margulis-Mohammadi-Oh \cite{MMO}. 

Now to the proof of Theorem \ref{main1}. There is a well established framework in place to prove this kind of result. From the flow $(\G \ba G/M, a_t)$ one first constructs a space $\hat U$ and an expanding map $\Pp: \hat U \rightarrow \hat U$. Associated to the expanding maps we then have a family of so-called transfer operators. In our case, because we want to study the compact group extension $\G \ba G$ over the geodesic flow $\G \ba G / M$, we are led to study transfer operators that have been twisted by representations of the compact group $M$. The well-worn arguments of Section \ref{transfer_bounds_are_enough_section} then establish that our exponential mixing result will follow as soon as we can prove certain spectral bounds for these transfer operators; see Theorem \ref{transfer_operator_bounds} and Proposition \ref{transfer_operators_are_engough}. 

At this stage we are left trying to prove the spectral bounds described in Theorem \ref{transfer_operator_bounds}. This type of bound was studied by Dolgopyat \cite{Do}, and we follow his machinery. The challenges of using his work in our case are firstly the fractal nature of the limit set $\Lambda(\G)$ (see \eqref{define_limit_set}) and secondly the twist in our transfer operators coming from the $M$-representation.

The necessary pre-requisites for Dolgopyat's arguments are addressed in Sections \ref{nli_section} and \ref{NCP_section}. The purpose of Section \ref{nli_section} is to assert that the summand terms in the definition of the transfer operator (see \eqref{define_transfer_operator}) are in some sense rapidly oscillating. This rapid oscillation is later used to show that there is cancellation amongst those summand terms at least on a ``reasonably dense'' set. 

The other pre-requisite for Dolgopyat's argument is a density statement for the limit set, establishing that the intersection of $\Lambda(\G)$ with any ``reasonably dense'' set is itself large. This is addressed in Section \ref{NCP_section}; it's analogous to the triadic intersection property discussed by Naud \cite{Na} and Stoyanov. 

Given the pre-requisites described in the previous two paragraphs, it might seem like we're almost done. We want spectral bounds for transfer operators: we know that the summand terms admit cancellation on a large set, and we know that that large set meets the limit set (where the interesting dynamics happens). One just needs to put those pre-requisites together to get spectral bounds. That synthesis turns out to be much more intricate than one might expect. Happily the difficulties at this stage are technical rather than conceptual; now that we have the pre-requisites framed appropriately, we can follow the original argument of Dolgopyat and complete the proof.

\subsection{Acknowledgements and Remarks} First and foremost I would like to thank Ralf Spatzier, wthout whom this paper could in now way have been completed. He was a large part of the development of the underlying ideas, provided very extensive technical help, and was a source of encouragement and good humour throughout. 

I am grateful to many others for technical advice and helpful conversations, amongst them Hee Oh, Michael Magee, Ilya Gekhtman, and Wenyu Pan.

I am particularly grateful to Mark Pollicott both for several useful conversations and for his patience during them. I do now understand that he and Richard Sharp have been pursuing similar arguments, though perhaps with different emphasis. 

\section{Notation and background} \label{sec2}
\subsection{Notation and the structure of $\SO(n, 1)$} We retain the notation of the introduction. $G$ will denote the identity component $\SO(n, 1)_+$ acting isometrically on hyperbolic $n$-space $\mathbb{H}^n$. We choose a base point $o \in \mathbb{H}^n$ and a unit tangent vector $w_0 \in \operatorname{T}_o\mathbb{H}^n$. Then the stabilizer $K$ of $o$ is a maximal compact subgroup of $G$ and is isomorphic to $SO(n)$. It contains the stabilizer $M$ of $w_0$, which is isomorphic to $\SO(n-1)$. We identify 
$$ G/ K = \mathbb{H}^n \mbox{ and } G/M = \operatorname{T}^1(\mathbb{H}^n).$$
We can choose a one paramater diagonalizable subgroup $A = \lbrace a_t: t\in \mathbb{R} \rbrace$ such that $M$ and $A$ commute, and such that the right $a_t$ action on $G/M = \operatorname{T}^1(\mathbb{H}^n)$ corresponds to unit speed geodesic flow. We denote by $N^+$ and $N^-$ the unstable and stable horospherical subgroups respectively:
$$ N^\pm = \lbrace h \in G : a_t h a_{-t} \rightarrow e \mbox{ as } t \rightarrow \pm \infty\rbrace .$$
They are commutative groups isomorphic to $\mathbb{R}^{n-1}$ and are normalized by $M$ and $AM$.

 The geodesic flow or $a_t$ action on $G/M$ is the canonical example of a hyperbolic dynamical system: the strong stable (unstable) manifolds for $w = gM$ are given by
 $$  W^{ss}(w) = wMN^-\mbox{ and }  W^{su}(w) = wMN^+$$
 respectively. The weak stable and unstable manifolds are
  $$  W^{ws}(w) = wMAN^-\mbox{ and }  W^{wu}(w) = wMAN^+.$$

\subsection{Convex cocompact subgroups of $G$} \label{convex_cocompact} We want to study discrete subgroups $\G < G$ via the dynamics (or representation theory) of $\G \ba G$. When $\G$  is a lattice (that is, when $\G < G$ has finite covolume) we have excellent tools available to do this \cite{EM, GN}, so our focus now is on groups that are Zariski dense but sparse, having infinite covolume in $\SO(n ,1)$.

It seems to be very difficult to understand sparse groups in full generality. Certain sub-classes, however, seem more tractable: the simplest of these are convex cocompact groups and, more generally,  geometrically finite groups (see \cite{bowditch93}). For now we focus on convex cocompact groups, as they are the ones most easily amenable to our proof techniques. 

A discrete Zariski dense torsion free subgroup $\G < G$ is said to be convex cocompact if the non-wandering set $\Omega$ for the geodesic flow on $\G \ba G /M$ is compact. We will observe the additional convention that convex cocompact subgroups should be non-lattices. The advantage of the convex cocompactness assumption for our purposes is that all the interesting dynamics for the geodesic flow happens on a compact set. As is widely known, this allows application of tools from the study of Anosov systems. 

\subsection{Limit sets and critical exponents}
For a  discrete  subgroup $\G < G$ we have the limit set 
\begin{equation} \Lambda = \Lambda(\G) = \lbrace \mbox{accumulations points of $\G o$ in $\partial (\mathbb{H}^n)$} \rbrace.\label{define_limit_set} \end{equation}
This is easily seen to be independent of the choice of base point $o \in \mathbb{H}^n$. It is typically a fractal subset of the boundary, whose Hausdorff dimension we denote by $\delta = \delta(\G) \in [0, n-1]$.
We can think of $\delta$ as a measure of how sparse $\G$ is in $G$: if $\G$ is very sparse then $\Lambda$ is small and $\delta$ is close to zero; if $\G$ is ``thick'' (almost a lattice) then $\Lambda$ is large and $\delta$ will be close to $n-1$, which is the Hausdorff dimension of $\partial \mathbb{H}^n$. 

Throughout the rest of the paper we will assume that

$$\G  < G \mbox{ is  convex cocompact and Zariski dense.} $$

\begin{remark}
For $\G$ convex cocompact and Zariski dense we actually have $\delta(\G) \in (0, n-1)$: this follows from existence of a unique Patterson-Sullivan density of dimension $\delta$ on $\Lambda(\G)$ in this case, and from finiteness of the measures of maximal entropy. These questions are discussed at greater length in Section \ref{matrix_coefficient_section}. 
\end{remark}

\subsection{Previous successes of a dynamical approach}  Our approach will be to view the geodesic flow on the non-wandering set $\Omega$ as a hyperbolic dynamical system and to apply the thermodynamic formalism. The general achievements of this viewpoint are too numerous to recount here, but there have in particular been many successes in areas related to our current interests: Lalley \cite{La} used these techniques to understand orbit counting questions for thin subgroups of $\SL_2(\mathbb{R})$; Naud \cite{Na} refined Lalley's work to give precise error estimates; Bourgain-Gamburd-Sarnak \cite{BGS} combined the thermodynamic formalism with expander graph techniques to explore number theoretic aspects of thin group orbits; Stoyanov \cite{St} used thermodynamic techniques to obtain exponential mixing of the geodesic flow on $\G \ba G / M$ for convex cocompact groups (amongst other systems).

In amongst these beautiful and powerful ideas we should emphasize the focus of our current discussion: we want to understand holonomies as in Theorem \ref{main3} with good error terms.  This requires a good quantitative understanding of $\G \ba G$ rather than $\G \ba G / M$, so we are working with compact group extensions of a hyperbolic system. The qualitative understanding of $\G \ba G$ is provided by work of Parry-Pollicott \cite{PP} or Brin \cite{Br}. The quantitative understanding of the base space $\G \ba G / M$ is provided by Stoyanov \cite{St} using ideas of Dolgopyat \cite{Do}. Our task is in some sense to provide a synthesis of the Parry-Pollicott/Brin and Stoyanov/Dolgopyat approaches. 

The current work could reasonably be regarded as a partner paper to \cite{OWhj}, which described similar ideas in the setting of holomorphic dynamics.

\subsection{Further connections} The aim of this paper is to describe how Dolgopyat's ideas can be used to give exponential decay of correlations for a compact group extension of a hyperbolic flow. In fact, though, the same ideas should give much more.

Firstly, one should be able to deduce a central limit theorem for the $a_t$ action from the results of this paper. 

Secondly, it is often useful for number theoretic applications to understand how the exponent in the error term of Theorem \ref{main1} changes on passage to a finite index subgroup of $\G' < \G$.  Equivalently one asks how the mixing rate changes on passage to a finite cover of our given convex-cocompact hyperbolic manifold. As in \cite{MOW}, one would expect to be able to answer these questions by combining our current discussion with an analysis of expander graph properties of the covering map data. By carrying out this synthesis one would expect to recover the sieve/number theoretic results of \cite{MO}.  

Finally, it is reasonable to hope that our current discussion should contribute to a proof of the non-spherical spectral gap conjecture of Mohammadi-Oh \cite[Conjecture 1.2]{MO}.

\section{Suspension space models and transfer operators} \label{transfer_bounds_are_enough_section}

Our aim is to use ideas from the thermodynamic formalism to understand representation theoretic properties of $\Gamma \ba G$. In this section we build the appropriate model for the $A$ action on $\G \ba G$, construct the required transfer operators, and describe how spectral bounds for these transfer operators will give the main results we want.  Experts in the thermodynamic approach will (hopefully) find few surprises in this section, but we will attempt an account for the uninitiated nonetheless. 

We first recall the basic properties  of Markov sections and outline a particular construction, with certain special properties.  These  will later be used to give extra control of oscillatory properties of smooth functions

\subsection{Markov sections} \label{Markov_sections} Under our running assumptions, the non-wandering set $\Omega$ of $\G \ba G / M$ is a compact $A$-invariant set, and forms a hyperbolic set for the flow. For a point $z\in \Omega$ and small sets $U$ (respectively $S$) contained in the strong unstable  manifold $W^{su}(z) \cap \Omega$ (respectively the strong stable manifold $W^{ss}(z) \cap \Omega$) through $z$ we have a rectangle given by 
$$  [U, S] = \lbrace [u, s] = \mbox{the unique local intersection of $W^{ss}(u) \cap W^{wu}(s)$}:u \in U, s \in S \rbrace.$$
Abusing notation slightly we write $[u, S] := [\lbrace u \rbrace, S]$ and $[U, s] := [U, \lbrace s \rbrace]$ for elements $u\in U$ and $s \in S$.

Given a finite collection of points $z_i\in \Omega$ with $1 \leq i \leq k_0$ and small pieces $z_i \in U_i$ (respectively $z_i\in S_i$) of $W^{su}(z_i) \cap \Omega$ (respectively $W^{ss}(z_i) \cap \Omega)$ we form the rectangles $R_i = [U_i, S_i] \ni z_i$. Write
$$U = \cup U_i, S = \cup S_i, R = \cup R_i $$
for the unions. We say that the collection of rectangles $\lbrace R_i, i = 1, \ldots, k_0\rbrace$ is a Markov section of size $\epsilon >0$ for the flow if they are disjoint and

\begin{itemize}
\item each of the sets $U_i, S_i, R_i$ is compact of diameter at most $\epsilon$;
\item each set $U_i$ is equal to the closure of it's interior $U_i^o$ in $\Omega \cap W^{su}(z_i)$;
\item each set $S_i$ is equal to the closure of it's interior $S_i^o$ in  $\Omega \cap W^{ss}(z_i)$;
\item the first return map $\poi: R := \coprod R_i \rightarrow R$ satisfies
\begin{enumerate}
\item $\poi([U_i^o, s]) \supset [U_j^o, s']$ and 
\item $ [u, S_i^o] \subset \poi^{-1}([u, S_j^o] ) $ 
\end{enumerate}
whenever $\poi([u, s]) = [u', s']$ with $u, u', s, s'$ in the interiors of $U_i, U_j, S_i, S_j$ respectively;
\item $\Omega = Ra_{[0, \epsilon]}$;
\item  for every pair $i \neq j$ either $R_i a_{[0, \epsilon]} \cap R_j = \emptyset$ or $R_j a_{[0, \epsilon]} \cap R_i = \emptyset$. 
\end{itemize}
 
The existence of Markov sections of arbitrarily small size is a consequence of ideas of Bowen \cite{Bowen73} and Ratner \cite{Ra}.  As written their arguments apply to the case of a compact manifold, however it is well understood that the same results hold in our setting. One reference is work of Pollicott \cite{Polli87}, who gave an account of Bowen's argument for Smale flows; the process is to start with a relatively general cross section to the flow, and then within that cross section to build rectangles satisfying the Markov properties.  The flexibility in choosing a cross section will be important to us since we will need to pick a cross section that lies in a smooth submanifold to utilize the smoothness properties of the functions under consideration. 

The Markov section provides a model for the flow as follows. We write $\tau :R \rightarrow \mathbb{R}^+$ for the first return time map (so that $\poi(x) = xa_{\tau(x)})$, and define the equivalence relation 
$$ (x, t + \tau(x)) \sim (\poi(x), t)$$
on $R\times \mathbb{R}$. We write
$$R^\tau := R\times \mathbb{R} / \sim.$$
Then the map $\pi: R \times \R \rightarrow \G \ba G / M$ given by $(x, t) \mapsto xa_t$ is a semi-conjugacy intertwining the flow $\mathcal{G}_s(x, t) := (x, t+s)$ and the geodesic flow. 

The model $R^{\tau}$ is unpleasant to work with; in particular the map $\poi$ is not continuous on boundaries $\overline{\poi^{-1}(R_i)} \cap \overline{\poi^{-1}(R_j)}$. To work around this we restrict our flow to a large subset of $R^\tau$. More precisely we write
$$ \hat{R} := \lbrace x \in R: \poi^k(x) \in \coprod  [U_j^o, S_j^o] \mbox{ for all } k \in \mathbb{Z}\rbrace.$$
This is a residual subset of $R$, and $\hat{R}^\tau = (\hat{R} \times \mathbb{R} / \sim )\subset R^\tau$ is therefore a flow invariant residual subset. The restriction of $\pi : R^\tau \rightarrow \G \ba G / M$ to $\hat{R}^\tau$ is then continuous and bijective. The image $\pi(\hat{R}^\tau)$ has full BMS measure, which makes this model suitable for addressing mixing questions. 

This isn't quite the model we will work with: firstly it's a model for geodesic flow rather than frame flow; secondly it is both possible and beneficial to replace the hyperbolic map $(\hat R, \poi)$ with an expanding map $(\hat U, \Pp)$; finally many of our later arguments will need a degree of smoothness, which isn't available here. We will address these refinements one by one in the rest of this section, but first we need to translate the BMS measure into this language. 

\subsection{Equilibrium states and the measure of maximal entropy} We refer to Chernov \cite[Section 4]{Che} for relevant background information. Both the discrete time system $(\hat R, \poi)$ and the continuous time system $(\hat R^\tau, \mathcal{G}_t)$ carry interesting measures. For a Lipschitz function $f$ on $\hat R$ define the pressure of $f$ to be the supremum
$$ Pr(f) = \sup_{\mu} \left( \mbox{entropy}_\nu(\poi) + \int f d\nu \right) $$
over all $\poi$-invariant probability measures on $\hat R$. In fact there is a unique probability measure $\nu_f$ that achieves this supremum; it is called the $f$ equilibrium state. For our purposes the interesting example is $\nu:= \nu_{-\delta \tau}$.

Our interest in $\nu$ comes from its connection to the measure of maximal entropy. We form a probability measure $\frac{1}{\int \tau d\nu } d\nu \times dt$ on $\hat R^\tau$. This is flow invariant, gives the unique measure of maximal entropy for the flow $\mathcal{G}_t$ on $\hat R^\tau$, and has measure theoretic entropy $\delta$. It follows that the normalized pushforward
 $$\frac{1}{\int \tau d\nu} \pi_* \left(d\nu \times dt \right) = m^{\BMS}$$
 coincides with the BMS measure on $\G \ba G$. 
\subsection{A model for frame flow} \label{A_model_for_frame_flow}
 The frame bundle $\G \ba G$ forms a principal $M$ bundle  over $\G \ba G /M$. We build a section $F$ of this bundle over $R$ as follows. First trivialize the bundle over each $z_i$ by choosing a frame at each point. We extend the trivialization  first to all of $U_i$ by demanding that our choices of frame over $u\in   U_i$ all be backward asymptotic to that over $z_i$. Similarly we extend the choice to all of $ R_i$ by demanding $F([u, s])$ and $F([u, s'])$ be forwards asymptotic whenever $u \in U_i$ and $s, s' \in S_i$.

\begin{definition}Similar to the first return time $\tau$ we have a holonomy map $\theta : R \rightarrow M$ with the property that
\begin{equation} 
F(x) a_{\tau (x)} = F(\poi(x)) \theta (x) ^{-1}
\end{equation} 
\end{definition}


\begin{lemma} The functions $\tau$, $\theta$ are constant along strong stable leaves.  
\end{lemma}

\begin{proof}
That $\tau$ is constant along strong stable leaves in $R_i$ comes precisely from the Markov property.  Also, $\theta$ is constant along strong stable leaves in $R_i$ by the properties of the section $F$.
\end{proof}

\begin{Nota}
We write $\Phi(x) := a_{\tau(x)}\theta(x) \in AM$; this will save space later. We will denote the ergodic sums and products by
$$ \tau^{(n)}(x) = \sum_0^{n-1}\tau(\poi^j(x)),  \theta^{(n)}(x) = \theta(x) \theta(\poi(x)) \ldots \theta(\poi^{n-1}(x))$$
and 
$$  \Phi^{(n)}(x) = \Phi(x) \Phi(\poi(x)) \ldots \Phi(\poi^{n-1}(x))$$
\end{Nota}

\begin{remark}
The functions $\Phi$ are strongly related to Brin-Pesin moves, as will be made precise in Lemma \ref{Brin_pesin_versus_path_maps}. The details of this correspondence arise out of our particular choice of Markov section $R$ and out of the choice of the framing $F$.  
\end{remark}

We  have a map $\hat R \times \R \to \Gamma \ba G/ M, (x, t) \mapsto x a_t$. We want to extend this to a map  $\hat R \times \R \times M \to \Gamma \ba G, (x, t, \theta) \mapsto F(x) a_t \theta$.  Then we have
$(x, \tau (x), e) \mapsto F(x) a_{\tau(x)}  = F(\poi(x)) \theta ^{-1} (x) = (\poi(x), 0,\theta ^{-1} (x))$. 
Hence we define an equivalence relation $\sim$ on $\hat R \times \R \times M$ such that 
$$(x, \tau(x), e)  \sim (\poi(x), 0, \theta ^{-1} (x)).$$

To summarize, we have  a  model for the frame flow: let
\begin{equation} \hat R^{\tau, \theta} := \hat R \times M \times \mathbb{R} / \sim \end{equation}
where $\sim$ is the relation generated by
\begin{equation} ([u, s], t + \tau(u), \theta) \sim (\poi([u, s]),t,  \theta^{-1} (u)\theta  ). \end{equation}
This system has an obvious flow which we shall denote 
$$\mathcal{G}_t(x, s, \theta) = (x, s+t, \theta).$$ 

\subsection{Semiflows, and expanding maps} Next we describe the usual method for replacing the hyperbolic map $\poi: R \mapsto  R$ with an expanding map. Denote 
$$\Pp  =   \mbox{proj}_S\circ \poi: U \rightarrow U $$
for the map given by first including $U$ into $R$, then applying $\poi$, and finally projecting back to $U$ along stable leaves. We want to think of $\Pp$ as an expanding map on $U$; unfortunately it is discontinuous as written, so we need a little more care. 

Like the residual subset $\hat R \subset R$ we have also a residual subset
$$ \hat U := \lbrace u \in U: \Pp^k(u) \mbox{ is an interior point of $U$ for all } k \in \mathbb{Z}_{\geq 0}\rbrace.$$
The point is that $\hat U$ is a large $\Pp$ invariant set on which $\Pp$ is continuous.

\begin{Nota} Abusing notation slightly, we will also write $\nu$ for the projection of the equilibrium state $\nu$ from $\hat R$ to $\hat U$. \end{Nota}
 
  \begin{definition} Let 
\begin{equation} \hat U^{\tau, \theta} = \hat U \times M \times \mathbb{R}_{\geq 0} / \sim \end{equation}
where $\sim$ now denotes the relation generated by 
\begin{equation} (u,  t + \tau(u), \theta) \sim (\Pp(u), t, \theta(u)^{-1} \theta ).\end{equation}
\end{definition}

This space also has a natural semi-flow which we also denote by $\mathcal{G}_t$. Exponential mixing of this semi-flow will imply exponential mixing of $\hat R^{\tau, \theta}$ by a hyperbolicity argument - see Lemma \ref{semi_flow_mixing}. On the other hand we will see later in this section that spectral bounds for certain transfer operators imply exponential 
mixing of $\hat U^{\tau, \theta}$.

Before that we must address one last technical point: a priori $\hat U$ is only a topological space, while our later arguments will require us to think of it as a smooth manifold. The point, of course, is that the complicated set $\hat U$ is naturally embedded into the union of the manifolds $z_iN^+$. Our next task is to make this precise.

\subsection{Smooth structure for $U, \theta$, and $\tau$} \label{smooth_structures} We want to do calculus on $U, \theta$, and $\tau$; this is messy, since $U$ is fractal in nature. We now recall how to embed everything into a smooth manifold and deal with that technicality.

\begin{Nota}
We will write $\mathrm{Tr}_{ij} = 1$ whenever $\poi[U^o_i, S^o_i]$ meets $[U^o_j, S^o_j]$, and $\mathrm{Tr}_{ij} = 0$ else.  This gives a $k_0 \times k_0$ matrix of zeros and ones which we call the transition matrix. 
\end{Nota}
We take this opportunity to fix one more piece of notation.
\begin{Nota} We say a sequence $i_1 \ldots i_l \in \lbrace 1 \ldots k_0 \rbrace^l$ is admissible if $\mathrm{Tr}_{i_j, i_{j+1}} = 1$ for each $j = 1 \ldots l-1$. For an admissible sequence $(i_1 \ldots i_l)$ we denote the cylinder 
$$ \cyl(i_1, \ldots i_l) := \lbrace u \in \hat U: \Pp^{j-1}(u) \in U_{i_j} \mbox{ for all } j = 1 \ldots l\rbrace.$$\end{Nota}

 While the rectangles $R_i = [U_i, S_i]$ are fractal sets, they nonetheless come embedded in smooth cross sections to the analytic flow. We can choose open neighbourhoods $\tilde U_i$ of $U_i$ in $W^{su}(z_i)$. We write $\tilde U$ for the union. This allows us to talk precisely about $C^1$ functions on fractal sets $U_i$: we will mean that the function is the restriction of a $C^1$ function on $\tilde U_i$. 

We write, for reasons that will become clear later, $\tilde U = \coprod \tilde U_i$ for the {\it disjoint} union of the $\tilde U_i$. The inclusion of $\hat U_i$ into $\tilde U_i$ allows us to think of $\nu$ as a measure on $\tilde U$.


 Note that $\Pp$ does not extend naively to a map on $\tilde U$; the expanding nature of the map precludes any possibility that the image would land in $\tilde U$. However the situation is much better for the inverse to $\Pp$. By the conditions on our Markov sections, and since $\Pp$ is eventually expanding, we may assume that the neighborhoods $\tilde U_i$ have the following properties: whenever $\mathrm{Tr}_{ij} = 1$ there is an inverse function $\Pp^{-1} = \Pp_{ji}^{-1} : \tilde U_j \rightarrow \tilde U_i$ given as the local intersection of $W^{su}(z_i) \cap W^{ws}(x)$. These inverses are eventually contracting in the following sense: for any admissible sequence $(i_1 \ldots i_k)$, the composition
$$ \Pp^{-1}_{(i_1 \ldots i_k)} :=  \Pp^{-1}_{i_1 i_2} \circ  \ldots \circ  \Pp^{-1}_{i_{k-1} \ldots i_k} : \tilde U_{i_k} \rightarrow \tilde U_{i_1}$$ 
satisfies, for any unit vector $v$,
\begin{equation} \frac{c_0}{ \kappa_1^k} < | \nabla  (\Pp^{-1}_{(i_1 \ldots i_k)} ) \cdot v  | < \frac{1}{c_0\kappa^k} \label{definec_0kappakappa_1} \end{equation}
for appropriate constants $c_0 \in (0, 1)$ and $1< \kappa < \kappa_1$ chosen independent of $k$. We will sometimes refer to a map of this form as a section for $\Pp^k$ defined on $U_{i_k}$.

Note that $\tau$ and $\theta$ extend to give smooth versions of the first return time and holonomy functions
$$\tau_{ij}: \tilde U_i \rightarrow \mathbb{R}, \theta_{ij} : \tilde U_i \rightarrow M$$
defined as the $A$ or $M$ displacement from $x \in \tilde U_i$ to $W^{su}(\Pp(x))$, with $\tau_{ij} = \tau$ and $\theta_{ij} = \theta$ on the interior of $U_i \cap \Pp^{-1}(U_j)$.

\begin{Nota} We write $\tilde U$ for the disjoint union of the $\tilde U_i$.\end{Nota}

\begin{remark} The functions $\tau_{ij}, \theta_{ij}$ really do depend on both $i$ and $j$ \end{remark}

\begin{remark} A priori the statements above seem clear under the assumption that the maps $\Pp^{-1}_{ij}$ are strictly contracting. However it is straightforward to modify the natural metric on $U$ so as to achieve strict contraction for $\Pp^{-1}$; one need only sum the metric over long and sensibly chosen orbits of $\Pp^{-1}$. 
\end{remark}


\subsection{Functions of fixed $M$-type} Our aim is to understand mixing for functions $\phi, \psi$ on $\G \ba G$. We think of $\G \ba G$ as a principal $M$ bundle over $\G \ba G / M$. The idea is to break $\phi, \psi$ down into ``Fourier types'' for $M$ and prove exponential mixing one type at a time. We need to describe how to do this, and we need to be careful that when we add the various types back together at the end the resultant sum will be convergent. We explain the required properties here.

  Consider the representation of $M$ on $L^2(M)$.  By the Peter Weyl theorem, $L^2(M) =  \oplus _{\mu} (\dim \mu) V_{\mu}$ where the summation is over irreducible representations of $M$.  Given a vector $v \in L^2(M)$, write $v = \sum v_{\mu}$ for its decomposition into isotypical vectors.  

\begin{lemma} \label{1st_fact_about_decay_of_fourier_modes}
Suppose $v \in L^2(M)$ is a smooth vector.  Then for any $m>0$:
\begin{equation}
 \| v_{\mu} \|  \leq c(\mu) ^{-m} (\dim \mu)^2 \|v\|_m    
\end{equation} 
Here $c(\mu)$ is the eigenvalue of the Casimir operator of $\mu$, and $|| \cdot||_m$ denotes the $C^m$ norm.  The Casimir eigenvalue is given by
\begin{equation}
c(\mu) = 1 + q(\lambda (\mu) + \rho) - q (\rho) 
\end{equation} 
where $\rho $ as usual is half the sum of the positive roots, $q$ is a positive definite quadratic form on the dual space of the maximal torus in $M$, and $\lambda$ is the highest weight vector for $\mu$.
\end{lemma} 

\begin{proof}
\cite[Lemma 4..4.2.2 and 4.4.2.3]{Warner} 
\end{proof}



\begin{Nota} For an $M$-representation $(\mu, V_\mu)$ and a real number $b$ we have an $MA$-representation $(\mu_b, V_\mu)$ by taking
$$ \mu_b(a_tm) \cdot v = e^{ibt} \mu(m) \cdot v.$$

\end{Nota}

\begin{Nota} Fix now a representation $\mu$ for $M$ acting unitarily on a complex vector space $V$ and a real number $b$. For any fixed unit vector $v \in V$ the orbit map $m \mapsto \mu_b(ma_t)\cdot v$ is Lipschitz; we write $||\mu_b||$ for the worst Lipschitz constant of any unit vector $v$. \end{Nota}

\begin{Nota} \label{define_d_0} We write $d_0 > 0$ for the minimal value of $||\mu_b||$ as $b$ ranges over $\mathbb{R}$ and $\mu$ ranges over {\it non-trivial} $M$ representations. Since the exponential mixing is already known at the level of the unit tangent bundle, we need only deal with these non-trivial $M$ representations.   \end{Nota}

One simple consequence of Lemma \ref{1st_fact_about_decay_of_fourier_modes}  (together with the Weyl dimension formula) is as follows.

\begin{corollary} \label{2nd_fact_about_decay_of_fourier_modes}
For any positive $m$ there exists $r_0\in \mathbb{N}$ and $C > 0$ such that
$$ ||v_\mu|| \leq \frac{C ||v||_{C^{r_0}}}{||\mu||^m}$$
for any smooth vector $v \in L^2(M)$. 
\end{corollary}

\subsection{Transfer operators of $M$-type $\mu$} As is usual we now want to reduce the study of mixing properties of $(\hat U^{\tau, \theta}, \nu\times d\theta \times dt, \mathcal{G}_t)$ to the study of spectral bounds for appropriate transfer operators. For $\xi \in \mathbb{C}$ and $(\mu, V)$ an irreducible representation, we define a transfer operator
\begin{eqnarray} {\mathcal{L}}_{\xi, \mu}: C(\hat U, V) &\rightarrow& C(\hat U, V)\\ ({\mathcal{L}}_{\xi, \mu}h)(u) &= &\sum_{\Pp(u') = u}e^{-\xi\tau(u')}
\mu ( \theta(u')) ^{-1}(h(u')).   \label{define_transfer_operator} \end{eqnarray}

Of course we also have a transfer operator for the trivial representation, which we write ${\mathcal{L}}_{\xi}$ acting on $C^1(\hat U, \mathbb{C})$. The most important of these is the operator ${\mathcal{L}}_{\delta}$, which actually acts on $C^1(\hat U, \mathbb{R})$. Ruelle's Perron-Frobenius Theorem \cite{Ru} tells us that ${\mathcal{L}}_{\delta}$ has a positive eigenfunction $h_\delta$ with eigenvalue $1$. It then turns out to be convenient to normalize our transfer operators using this eigenfunction by taking
\begin{equation}  \hat{\mathcal{L}}_{\xi, \mu}(h) = \frac{{\mathcal{L}}_{\xi, \mu}(hh_\delta) }{h_\delta}. \label{defineh_delta}   \end{equation} 
The point of this normalization is that we get to assume that $\hat{\mathcal{L}}_{\delta}(1) = 1$. It also provides the useful by-product that $\nu$ is an $\hat{\mathcal{L}}_{\delta}$ eigenmeasure with eigenvalue $1$.

Unfortunately the transfer operators on $\hat U$ are somewhat difficult to work with directly; we want to apply $C^1$ arguments, but $\hat U$ is a fractal set. This was the reason for introducing the open set $\tilde U$. We may therefore think of transfer operators as acting on 
\begin{eqnarray}\tilde{\mathcal{L}}_{\xi, \mu}: C^1(\tilde U, V) &\rightarrow& C^1(\tilde U, V)\label{here_is_your_transfer_operator} \\ (\tilde{\mathcal{L}}_{\xi, \mu}h)(u) &= &\sum_{\mathrm{Tr}_{ij}}e^{-\xi \tau_{ij}(\Pp_{ij}^{-1}u)}
\mu ( \theta_{ij}(\Pp_{ij}^{-1}u)) ^{-1}(h(\Pp_{ij}^{-1}u))   \end{eqnarray}
on the $\tilde U_j$ piece. We will suppress the $ij$ subscript when it seems unlikely to cause confusion.  
The inclusion map of $C^1(\tilde   U, V)$ into $C(\hat U, V)$ intertwines the operators $\hat{\mathcal{L}}$ and $\tilde{\mathcal{L}}$ and is an isometry for the $L^2(\nu)$ norm. 

The main technical result of our argument is the following spectral bound.

\begin{theorem} \label{transfer_operator_bounds}
There are positive constants $C, \eta$, and a constant $\rho \in (0, 1)$ with the following properties: for any non-trivial irreducible representation $\mu, V_\mu$ of $M$, any $b \in \mathbb{R}$, any $n \in \mathbb{N}$, any $\Re(\xi) > \delta - \eta$, and any $h \in C^1(\tilde U, V_\mu)$ we have 
$$ || \tilde{\mathcal{L}}_{\xi, \mu}^n h||_{L^2(\nu)} \leq C(1 + ||\mu||)^C \rho^n ||h||_{C^1}.$$
\end{theorem}

\subsection{Relating the flows to transfer operators} Fix $\phi, \psi$ smooth compactly supported functions on $\G \ba G$. We write $\nu$ for the $-\delta \tau$ equilibrium state on $R$. The embedding of $(R^{\tau, \theta}, d\nu \times dt \times dm) $ into $(\G \ba G, m^{\BMS})$ is measure preserving. We will therefore consider $\phi, \psi$ as functions on $R^{\tau, \theta}$ and aim to prove exponential mixing there. The next step is to convert $\phi, \psi$ into functions on $\hat U$ (or $\tilde U$). We do this essentially by integrating out the $S$ component. for a positive number $r$ and $(u, m, t) \in \tilde U_i\times M\times [0, \tau(u))$ choose
$$ \phi_r(u, m, t) = \int_{ S_i} \phi([u, s]ma_{t+r}) d\nu_u(s),$$
where $d\nu_u$ is the conditional measure on $ S_i$ for $d\nu$, which is given by the Patterson-Sullivan density. Note that the dependence of $\nu_u$ on $u$ is $C^1$; this follows because the measure of maximal entropy on $\G \ba G$ is a quasi-product measure, and is essential to provide required regularity of $\phi_r$. 
The precise relationship to the BMS measure is given by the following lemma (modeled on \cite{AGY}).
\begin{lemma} \label{semi_flow_mixing}
We have the estimate
\begin{multline} \left|  \int_{\G\ba G}\phi(ga_{2t} )\psi(g) dm^{\BMS}(g)  -\frac{1}{\nu(\tau)} \int_{U^{\tau, \theta}}  \phi_t(u, m, r + t) \psi_0(u, m, r) dm dr d\nu \right| \\ \leq ||\phi||_{C^1} ||\psi||_{\infty} e^{-t}.   \end{multline}
\end{lemma}
Thus Theorem \ref{main1} will follow as soon as we establish exponential decay of 
\begin{equation} \frac{1}{\nu(\tau)} \int_{U^{\tau, \theta}}  \phi_t(u, m, r + t) \psi_0(u, m, r) dm dr d\nu; \label{reduced_to_U} \end{equation}
this is a question which takes place entirely on $\hat U$. Note at this stage that $\psi_0$ is actually the restriction of a $C^1$ function on $\tilde U$. Note also that the $M$ integral is clearly zero unless $\phi_t, \psi_0$ have the same $M$ type, $\mu$, say, so it is natural to break the integral up into a sum of $M$ types.

\begin{proposition} Theorem \ref{transfer_operator_bounds} implies exponential decay of \eqref{reduced_to_U} and hence Theorem \ref{main1}. \label{transfer_operators_are_engough}
\end{proposition}
A readable model argument is presented in the paper of Avila, Gouezel, Yoccoz. We recount their argument here for the reader's convenience.
By decomposing $\phi, \psi$ into $M$ types we may assume that each is of the same type $(\mu, V_\mu)$. It follows then that $\phi_t, \psi_0$ are of type $\mu$. For ease of notation we will write $ f = \psi_0$ and $g = \phi_t$, which are both elements of $L^{\infty}(\hat U^{\tau, \theta})$ in the first instance. We have functions $\hat{f}, \hat{g} \in C^1(\hat U^\tau, V_\mu)\subset C^1(\hat U^\tau, L^2(M))$ defined by the property
$$f(x, t, \theta) = \hat{f}(x, t)(\theta)$$
$$g(x, t, \theta) = \hat{g}(x, t)(\theta)$$
for $\theta \in M$. We write 
\begin{eqnarray*} \check{\rho}_{f, g}(t) &:=&  \int_{{\hat U}^{\tau, \theta}} f(u) \overline{g(\mathcal{G}_t(u)) }\\
 &=& \int_{\hat U} \int_0^{\tau(u)}\int_M f(u, s, \theta)\overline{  g(u, s + t, \theta)} d\theta ds d\nu\\
 &=& \rho(t) + \rho_1(t).\end{eqnarray*}
 where 
 $$  \rho(t)   :=  \int_{\hat U} \int_{\max(0, \tau(u) - t)}^{\tau(u)}\int_M f(u, s, \theta)\overline{  g(u, s + t, \theta)} d\theta ds d\nu $$
 and 
 $$  \rho_1(t)   :=  \int_{\hat U} \int_{0}^{\max(0, \tau(u) - t)}\int_M f(u, s, \theta)\overline{  g(u, s + t, \theta)} d\theta ds d\nu $$
 is compactly supported on $t < ||\tau||_\infty$. It turns out (think about the fact we are taking a one sided Laplace transform) to be both sufficient and easier to study $\rho(t)$.
 At this stage we recall the Laplace transform: for $\Re(\xi) > 1$, say, we have 
 \begin{eqnarray*} \mathcal{F}\rho(\xi) := \int_0^\infty\rho(t) e^{-\xi t} dt \end{eqnarray*}
 The key relationship between mixing and transfer operators is described in terms of the Laplace transform as follows: for a function $\hat f$ on $\hat U^\tau$ and a complex number $\xi$ we define a function $f_\xi$ on $\hat U$ by
 $$ \hat f_\xi(u) = \int_0^{\tau(u)}e^{-\xi t} \hat f(u, t) dt $$
 and similarly for $\hat g_\xi$. 
 \begin{lemma} \label{flows_and_transfer_operators} We have the relation  
\begin{equation}  \mathcal{F}\rho(\xi) = \sum_1^{\infty} \int_U \left\langle \hat{\mathcal{L}}^k_{\delta + \xi, \mu} \hat{ f}_{-\xi}(v),\hat{ g}_{\xi}(u)  \right\rangle_M d\nu. \label{as_an_infinite_sum} \end{equation}
 
\end{lemma}  
\begin{proof} We simply calculate: 
\begin{eqnarray*}\mathcal{F}\rho(\xi)  &=& \int_0^\infty \int_{\hat U} \int_{\max(0, \tau(u) - t)}^{\tau(u)}\int_M f(u, s, \theta)\overline{  g(u, s + t, \theta)} e^{-\xi t} d\theta ds d\nu dt\\
 &=&  \int_{\hat U} \int_{0}^{\tau(u)} \int_{\tau(u)}^\infty \int_M f(u, s, \theta)\overline{  g(u, t, \theta)} e^{-\xi (t - s) } d\theta dt ds d\nu\\
 &=&  \int_{\hat U} \int_{0}^{\tau(u)} \sum_1^{\infty} \int_{\tau^{(k)}(u)}^{\tau_{k+1}(u)} \int_M f(u, s, \theta)\overline{  g(u, t, \theta)} e^{-\xi (t - s) } d\theta dt ds d\nu\\ 
  &=&  \int_{\hat U} \int_{0}^{\tau(u)} \sum_1^{\infty} \int_{0}^{\tau(\sigma^{k}u)} \int_M f(u, s, \theta)\overline{  g(u, t + \tau^{(k)}(u), \theta)} e^{-\xi (t - s + \tau^{(k)}(u) ) } d\theta dt ds d\nu\\ 
 &=&  \int_{\hat U} \int_{0}^{\tau(u)} \sum_1^{\infty} \int_{0}^{\tau(\sigma^{k}u)} \int_M f(u, s, \theta)\overline{  g(\sigma^k(u), t, \theta^{-1}_k(u) \cdot \theta)} e^{-\xi (t - s + \tau^{(k)}(u) ) } d\theta dt ds d\nu\\ 
  &=&  \int_{\hat U} \int_{0}^{\tau(u)} \sum_1^{\infty} \int_{0}^{\tau(\sigma^{k}u)} \int_M f(u, s, \theta)\overline{ \theta^{(k)}(u)  g(\sigma^k(u), t,  \theta)} e^{-\xi (t - s + \tau^{(k)}(u) ) } d\theta dt ds d\nu\\ 
  &=&  \int_{\hat U} \int_{0}^{\tau(u)} \sum_1^{\infty} \int_{0}^{\tau(\sigma^{k}u)} \left\langle\hat{ f}(u, s), \theta^{(k)}(u) \hat{ g}(\sigma^k(u), t)  \right\rangle_M  e^{-\xi (t - s + \tau^{(k)}(u))}dt ds d\nu\\   
&=&   \sum_1^{\infty} \int_{\hat U} \left\langle\hat{ f}_{-\xi}(u), \theta^{(k)}(u) \hat{ g}_{\xi}(\sigma^k(u))  \right\rangle_M  e^{-\xi  \tau^{(k)}(u)}d\nu\\   
  \end{eqnarray*}
  with $\hat f_\xi$ and $\hat g_\xi$ defined as above.  Now we recall that $\nu$ is an eigenmeasure for the normalized transfer operator $\hat{\mathcal{L}}_{\delta}$ acting on $C(\hat U)$. Applying this $k$ times to the right hand side above we have 
 \begin{eqnarray*}  \mathcal{F}\rho(\xi) &=& \sum_1^{\infty} \int_{\hat{U}} \frac{1}{h_\delta(u)} \sum_{\sigma^kv = u}h_\delta(v)  \left\langle\hat{ f}_{-\xi}(v), \theta^{(k)}(v) \hat{ g}_{\xi}(u)  \right\rangle_M  e^{-(\delta + \xi)  \tau^{(k)}(v)}d\nu  \end{eqnarray*}
 for $h_\delta$ the lead eigenfunction of $\mathcal{L}_{\delta}$. We quickly recognize this as 
  \begin{eqnarray}  \mathcal{F}\rho(\xi) &=& \sum_1^{\infty} \int_{\hat{U}} \left\langle \hat{\mathcal{L}}^k_{\delta + \xi, \mu} \hat{ f}_{-\xi}(v),\hat{ g}_{\xi}(u)  \right\rangle_M d\nu \end{eqnarray}
  as expected.
  \end{proof}
  
  \begin{proof}[Proof of Proposition \ref{transfer_operators_are_engough}] Consider $\phi$ and $\psi$ smooth functions with compact support on $\G \ba G$. We note that all the results of this paper are well understood for $M$-invariant functions \cite{St}, so we may assume without loss of generality that $\int_M\phi(gm)dm = \int_M\psi(gm)dm = 0$ for all $g \in \G \ba G$; in other words we assume that the decomposition of $\phi, \psi$ into $M$ types has zero in the trivial part (see subsection \ref{talk_about_M_inv_functions}). Write $C$ for the constant from Theorem \ref{transfer_operator_bounds}.  As before we write $f = \psi_0$ and $g = \phi_t$. We can then decompose $f = \sum_\mu f_\mu$ and $g = \sum g_\mu$. pick $m$ such that $\sum_\mu (1 + ||\mu||)^{C - m}$ is summable. As in Corollary \ref{2nd_fact_about_decay_of_fourier_modes} we may choose $r_0$ such that the $\mu$-component $\hat g_\mu : (\hat U^\tau, V_\mu)$ satisfies 
  $$ |\hat g_\mu| \ll \frac{ ||\phi||_{C^{r_0}}}{||\mu||^m}$$
  pointwise. In fact more is true; for each $u \in \hat U$ we have the function
  $$\hat g_\mu(u, \cdot) : [0, \tau(u)) \rightarrow V_\mu.$$ By an elementary argument (see \cite[Defintion 7.2 and the proof of Lemma 8.3]{AGY}) the total variation of this function is bounded above by 
  $$|\hat g_\mu(u, \cdot)|_{\mbox{tot. var.}} \ll\frac{||\phi||_C^{r_0+ 1}}{||\mu||^m}.$$  
  As is well known, a bound on the total variation leads to decay estimates on the associated Laplace transform; in our case we have 
    \begin{equation}  |\hat g_{\mu, \xi}| \ll \frac{ ||\phi||_{C^{r_0 + 1}}}{||\mu||^m(1 + |\Im(\xi)|)}\end{equation} 
    as in \cite[Lemmas 7.19 and 7.20]{AGY}; we've denoted the imaginary part of $\xi$ by $\Im(\xi)$. 

Similarly we break $\hat f_\xi$ into $\mu$-pieces$\hat f_{\xi, \mu}$. Now we want to use Theorem  \ref{transfer_operator_bounds}  to bound 
  $$ || \hat{\mathcal{L}}^k_{\delta + \xi, \mu} \hat{ f}_{\mu, -\xi}||_{L^2(\nu)}.$$
  In principle we're worried that $\hat f_{\mu, -\xi}$ is defined a priori on $\hat U$, whereas Theorem \ref{transfer_operator_bounds} really uses the smooth structure. However it is straightforward to see that $F_{\mu, -\xi} = \hat{\mathcal{L}}_{\delta + \xi, \mu} \hat{ f}_{\mu, -\xi}$ is the restriction of a function $\tilde F_{\xi, \mu} \in C^1(\tilde U, V_\mu)$ satisfying 
   \begin{equation} ||\tilde{F}_{\mu, \xi}||_{C^1}\ll \frac{||\psi||_{C^2}}{1 + |\Im(\xi)|}.\end{equation}
  Theorem \ref{transfer_operator_bounds} then gives 
    $$ || \hat{\mathcal{L}}^n_{ \delta + \xi, \mu} \hat{ f}_{-\xi}||_{L^2(\nu)} \ll \frac{C (1 + ||\mu||)^C\rho^n||\psi||_{C^2}}{1 + |\Im(\xi)|}.$$
  From Lemma \ref{flows_and_transfer_operators} we see that $\mathcal{F}\rho(\xi)$ has holomorphic extension to the half plane with real part $\Re(\xi) >  - \eta$ and satisfies the bound 
  $$\mathcal{F}\rho_{f_\mu, g_\mu} (\xi)  \ll \frac{C(1 + ||\mu||)^C||\phi||_{C^{r_0+ 1}} ||\psi||_{C^2}}{||\mu||^m(1 + |\Im(\xi)|)^2}$$
  along the line $\Im(\xi) =  - \eta / 2$ whenever $\mu$ is non-trivial. We now want to apply some form of Paley-Wiener theorem to conclude that $\rho_{f_\mu, h_\mu}$ has good decay. In our case we can simply apply the Laplace inversion formula to see that 
  $$ \rho_{f_\mu, h_\mu}(t) = \frac{1}{2\pi i} \int_{-\infty}^\infty e^{(is - \eta/2)t}\mathcal{F}\rho_{f_\mu, g_\mu} (is - \eta/2)  ds.  $$
  Since the integral is absolutely convergent we may pull out the $e^{-t/2}$, which provides the required decay. We therefore have
 $$\rho_{f_\mu, g_\mu} (t) \ll ||\phi||_{C^{r_0+ 1}} ||\psi||_{C^2} (1 + ||\mu||)^{C-m} e^{-\eta t/2}.$$
 This provides our result in the case that $\phi, \psi$ are pure functions of type $\mu$. Summing over types we have our conclusion. 
\end{proof}

\subsection{A remark on $M$-invariant functions} As previously indicated, the results of this paper are already well understood for $M$-invariant functions. The argument uses the same framework as our current discussion, and was carried through by Stoyanov \cite{St}. For the readers convenience, however, we will comment briefly on the changes necessary to this paper if one wanted to study $M$-invariant functions as well. 

Consider now a pair of $M$-invariant functions $\phi, \psi$ on $\G \ba G$. We retain the notation above. Theorem \ref{transfer_operator_bounds} (the main estimate on spectral bounds for transfer operators) is stated for the case where $\mu$ is non-trivial. In fact the same result also holds when $\mu$ is trivial, so long as we assume $|\Im(\xi)| > 1$. This, together with the arguments of the current Section, tells us that that the Fourier transform $\mathcal{F}\rho_{\phi, \psi}$ is holomorphic on $|\Im(\xi)| > 1, \Re(\xi) > - \eta$. We then appeal to the complex Ruelle-Perron-Frobenius Theorem, which ends up saying that $\mathcal{F}\rho_{\phi, \psi}$ is meromorphic on $\Re(\xi) >  - \eta', |\Im(\xi)| < 1$ with just one simple pole at $\xi = 0$ ($\eta'$ here is just some small positive constant). Putting those two together we can then apply the Laplace inversion formula as above; the main term of the correlation function comes from the pole at $\xi = 0$, while the exponential error term comes from the otherwise pole free region $\Re(\xi) > \delta - \min(\eta, \eta')$.

\section{Non-Concentration of the limit set and doubling properties} \label{NCP_section}
\subsection{Non-concentration} The previous two sections describe how mixing properties for frame flow follow from spectral bounds for transfer operators. Now we want to prove such spectral bounds. Dolgopyat provides a framework in which to do this, but we must first set some preliminaries in place. 

The mantra of Dolgopyat's argument is that oscillation leads to cancellation, while cancellation leads to spectral bounds. We concentrate on the first part for now. The idea is that when we add up a collection of rapidly oscillating functions then we should expect some cancellation in the sum. Of course this is not quite true (think of the rapidly oscillating functions $ \cos nx$, which have no cancellation at $x = 0$); rather we should expect only to get cancellation at {\it most} points in the domain. The outcomes of this section will be used to show that, in our case, the large set where we get cancellation does actually meet the limit set, where the dynamics of the flow happen.

We identify the boundary $\partial \mathbb{H}^n$ of hyperbolic space with $\mathbb{R}^{n-1} \cup \lbrace \infty \rbrace$. Without loss of generality (replacing $\G$ by a conjugate if necessary) we assume that the limit set $\Lambda$ of $\G$ does not contain $\infty$, so we think of $\Lambda$ as a compact subset of $\mathbb{R}^{n-1}$. We further choose the one parameter subgroup $A = \lbrace a_t : t \in \mathbb{R} \rbrace$ to be the one parameter subgroup acting as expansion by $e^t$ on the upper half space $\mathbb{H}^n$. In this setting the stable horospherical group, isomorphic to $\mathbb{R}^{n-1}$, simply acts by translation on the boundary less infinity. The unstable horospherical group $N^+$, of course, is also isomorphic to $\mathbb{R}^{n-1}$, but acts in a somewhat more complicated way.

\begin{lemma}[Non-Concentration Property] \label{NCPlemma}
There exists $\delta > 0$ with the following property. For any $\epsilon \in (0, 1)$, any unit vector $w$ in $\mathbb{R}^{n-1}$ and any $x \in \Lambda$ there exists $y \in \Lambda$ such that 
\begin{itemize}
\item $y \in B_\epsilon(x)$, and
\item $|\langle y - x, w \rangle | \geq \epsilon \delta$.  
\end{itemize}
The same property holds if we require $x, y$ to both be in some fixed sub-cylinder of the limit set. 
\end{lemma}

\begin{remark}
This lemma is a statement about the affine properties of the limit set in one particular model for hyperbolic space. Of course there are many different affine geometries we could choose for the boundary of hyperbolic space, but it is clear from the formulation that the statement (if not the constants) of this lemma remains true if we chance to a different affine geometry; in fact we shall use this statement for the affine geometry on the piece on $\tilde U_1$ coming from the identification of $\tilde U_1$ with a small piece the unstable horospherical group $F(z_1)N^+$. 
\end{remark} 
This will be a straightforward result of a Sullivan type argument. 
\begin{proof}
Suppose the lemma fails. Then we may extract sequence $x_i \in \Lambda, \epsilon_i \in (0, 1)$, and $w_i \in \mathbb{R}^{n-1}$ for which
$$|\langle y - x_i, w_i \rangle | < \epsilon_i / i$$
for all $y \in B_{\epsilon_i}(x_i) \cap \Lambda$. Without loss of generality we may also assume that $x_i$ converges to $x\in \Lambda$, and that $w_i$ converges to a unit vector $w$. Since $\G$ is assumed Zariski dense, we know that the limit set $\Lambda$ is not contained in any hyperplane, and so that $\epsilon_i \rightarrow 0$ (in fact we need that the intersection of $\Lambda$ with any ball is not contained in a hyperplane, but the argument is similar). 

Now we want to use self similarity of the limit set. This relies (and indeed must rely) on the fact that $\G$ is convex cocompact. The group element $a_t$ acts on $\mathbb{R}^{n-1}$ as expansion by $e^t$. We write $n_x \in N^-$ for the contracting horospherical element sending $0$ to $x$. We choose $h_{y(x_i)} \in N^+$ to be an unstable horospherical element such that $n_{x_i} h_{y(x_i)}$ has forward end point $y \in \Lambda$. Note that the subset $h_{y(x_i)}$ is compact, and that $n_{x_i} h_{y(x_i)}$ is in the non-wandering set for the $a_t$ flow. Thus we may choose compact subsets $\Omega_1, \Omega_2 \subset G$ with the property that 
\begin{itemize}
\item $n_{x_i} h_{y(x_i)} a_t \in \Gamma \Omega_1$ for all $i$ and all $t\in\mathbb{R}$, and hence
\item $n_{x_i} a_{-t} \in \Gamma \Omega_2$ for all $t > 0$. 
\end{itemize}
Let $t_i = -\log(\epsilon_i)$ and choose $g_i \in \Omega_2, \g_i \in \G$ such that $n_{x_i}a_{-t_i} = \gamma_i g_i$. Then $g_ia_{t_i}n^{-1}_{x_i} = \gamma_i^{-1}$.  In other words the map $g_ia_{t_i}n^{-1}_{x_i}$ preserves the limit set. Let $z_i = g_i(0)$ be the image of $x_i$ under this map. Without loss of generality we have that 
$$g_i \rightarrow g\in \Omega_2 $$
and write $z = g(0)\in \Lambda$.  The failure of our lemma implies that
$$\Lambda \cap g_i(B_1(0)) \subset g_i(\lbrace z: |\langle w_i, z\rangle |< 1/i\rbrace).$$
Taking limits we see that
$$\Lambda \cap g(B_1(0)) \subset g(\lbrace y: \langle w_i, y\rangle= 0\rbrace).$$
In other words (an open subset of) the limit set is contained in (the smooth image of) a hyperplane. That's a contradiction.

\end{proof}

\subsection{Doubling properties for cylinders} For each cylinder $\hat U_i \subset \tilde U_i$ we have a the associated measure $\nu_i$. It is essential for later arguments that these measures $\nu_i$ satisfy a doubling property. 

\begin{lemma} \label{doubling_lemma}
There is a constant $C_1 > 0$ with the following property: for any $i$, any $x \in \hat U_i$ and any positive $\e$ we
$$ \nu_i(B_{2\e} (x) ) \leq C_1 \nu_i(B_{\e} (x) ) . $$
\end{lemma}
We will delay the proof of this lemma until section \ref{matrix_coefficient_section}, when more notation is available.

\section{Non-Local-Integrability and Brin-Pesin moves} \label{nli_section}

We return to the first part of our mantra: oscillation leads to cancellation. To apply this in our case we need to show that the summand terms of \eqref{here_is_your_transfer_operator} really are rapidly oscillating relative to one another. That is our aim in the current section. 

The formal statement we'll prove is Lemma \ref{NLI}, which is a version of Non-Local-Integrability (NLI) condition appropriate to the study of compact group extensions. Our argument is a refined version of a similar discussion from \cite{Do}, who discussed the case of trivial $M$-group. 

\subsection{Taxicab paths in rectangles and Brin-Pesin moves}
Consider a rectangle $R_1 = [U_1, S_1] \subset \G \ba G$ which we assumed small inasmuch as its diameter is less than the injectivity radius of $\G$ divided by $1000$. For a fixed base point $x_0 \in R_1$ we have the taxi cab path set:
$$P_{R_1, x_0} = \left\lbrace\mbox{ sequences } (x_0, x_1, x_2, \ldots, x_{r-1}, x_r = x_0): \begin{array}{l}   \mbox{ for every $j$ we have either}\\ \mbox{$x_{j+1} \in [x_j, S_1]$ or $x_{j+1} \in [U_1, x_j]$}\end{array} \right\rbrace. $$

This group has a natural family of maps into $MA$, given by Brin-Pesin moves. Choose the a lift $\check R = [\check U, \check S]$ of $R_1$ to the universal cover $G/M$. Choose $g$ such that $gAM \cap \check R = \check x_0$, the lift of $x_0$ to $\check R$. Given a sequence $(x_0, \ldots x_r = x_0) \in \hat P_{x_0, R_1}$ we lift it to a sequence $( \check{x}_0, \ldots , \check{x}_r)$ in $\check R$. We then associate a sequence $(g_0 = g, g_1, g_2, \dots g_r \in gAM)$ by requiring
\[ g_{i+1} \in g_iN^+ \mbox{ and } g_{i+1} MA \ni \check x_{i+1} \mbox{ if } x_{i+1} \in [U_1, x_i]\]
or 
\[ g_{i+1} \in g_iN^- \mbox{ and } g_{i+1} MA \ni \check x_{i+1}\mbox{ if } x_{i+1} \in [x_i, S_1].\]
We note then that $g_r \in gMA$. We denote their difference by 
\[  B_{g, \check R}(x_1, \ldots x_r) := g_r^{-1}g \in AM.\] 
We note that this function $\hat B$ is independent of choice of lift $\check R$, so we will suppress that part of the notation. We then have a function
\[  B_{g, R_1}:  P_{R_1, x_0} \rightarrow MA,\]
which is clearly related to Brin-Pesin moves. 

It will be useful to think of $B_{g, r_1}$ as a function on $N^- \times N^+$, which we may  do as follows. Recall the base point $z_1 \in R_1$ from subsection \ref{Markov_sections}. We can  identify $U_1$ with a small piece of $N^+$ around the origin, by sending $h \in N^+$ to $F(z_1)h$ (or more precisely to the projection of that in the unit tangent bundle). Similarly we identify $S_1$ with a small piece around the origin of $N^-$. In this way we may think of the Brin-Pesin homomorphism as giving a map from (some subset of) $N^+ \times N^-$ to $AM$ by sending  
$$ \Xi: (h, n) \mapsto B_{F(z_1)}(z_1, n, [h, n], h, z_1).$$
In practice we'll be interested in the case where $n$ is fixed at some nontrivial value, and $h$ varies; we therefore think of $\Xi$ principally as a function $N^+ \rightarrow MA$ defined by 
$$\Xi_n(h) := \Xi(h, n).$$

Now we relate this to $\tau, \theta$ and the NLI condition. Suppose we have a section $v$ of $\sigma^k$ defined on $\tilde U_1$ and taking values in $U_{i}$ respectively; formally here we mean, in the language of subsection \ref{smooth_structures},  
$$v = \Pp^{-1}_{(i_1, i_2 \ldots i_{k}, 1)} $$
where $(i_1, i_2 \ldots i_{k}, 1)$ as an admissible sequence. Associated to the section $v$ we have an element $ n = n(v) \in N^-$ given by
\begin{equation} F(\poi^k(v(z_1))) = F(z_1)  n. \label{horospherical_elements_for_sections}\end{equation} 

For $u\in U_1$ write $\phi(u) = \phi_{v, k}(u) := a_{\tau^{(k)}(v(u))} \theta^{(k)}(v(u)) \in MA$. We need to say that $\phi$ is in some sense rapidly oscillating. We do this by relating it to the $\Xi$ and the associated Lie theory as follows. 
\begin{lemma} \label{Brin_pesin_versus_path_maps} Let $n = n(v)$, and suppose that $u \in U_1$ with  $F(u) = F(z_1)h$ for some $h\in N^+$. Then
$$ \phi^{-1}(z_1) \phi(u) = \Xi_{ n}(h).$$
\end{lemma}
\begin{proof} 
Let $ s \in S_1$ be such that $\poi^k(v(z_1)) = \tilde s$. Note that under the identifications $N^-$ with $S_1$ and $N^+$ with $U_1$ we have $``h = u"$ and $``n = s"$. Choose group elements
\begin{itemize}
\item $n' \in N^-$ such that $F([u,  s]) = F(u)n'$, and
\item $h' \in N^+$ such that $F(v(u')) = F(v(u))h'.$
\end{itemize} 
Choose a lift $\check R_1$ of $R_1$ to the universal cover, and let $g$ be the associated lift of $F(z_1)$ to $G$. We now calculate $\Xi (h,  n) = B_{g}(z_1, n, [h, n], h, z_1)$. For this we simply take the sequence of group elements
\[ g_0 =  g\]
\[ g_1 =  g_0 n = F(v(z_1))\phi(z_1) \]
\begin{multline*}      g_2 =  g_1 \hat h' = F(v(z_1))h'\phi(z_1) \\ = F(v(u)) \phi(z_1) = F([u,  s]) \phi^{-1}(u) \phi(z_1)   \end{multline*}
for $\hat h' = \phi^{-1}(z_1)h'\phi(z_1) \in N^+$ the appropriate conjugate of $h'$. Next
\[ g_3 =g_2(\hat{n}')^{-1} =  F([u,  s])(n')^{-1} \phi^{-1}(u) \phi(z_1)  = F(u)\phi^{-1}(u) \phi(z_1)  \] 
for $\hat n' \in N^-$ the appropriate conjugate of $n'$. Finally we have 
\[ g_4 = g_3 (\hat h)^{-1} = F(u)h^{-1} \phi^{-1}(u) \phi(z_1)  = F(z_1)  \phi^{-1}(u) \phi(z_1)  \]
for $\hat h$ the appropriate conjugate of $h$. This gives 
\[ B_{g}( z_1, n, [h ,n], n, z_1)= g_4^{-1}g_0 =  \phi^{-1}(z_1) \phi(u). \]
\end{proof}
The Non-Local-Integrability condition is supposed to say that $\phi$ is in some sense rapidly oscillating. The relationship to $\Xi$ will allow us to see this as an easy consequence of Lie theory.

\begin{Nota}
For a function $f:M_1 \rightarrow M_2$ between Riemannian manifolds we write $\nabla f$ for the derivative map, and $||f||_{C^1}$ for the supremum of the operator norms of $(\nabla f)_x: T_xM_1 \rightarrow T_{f(x)}M_2$ with $x \in M_1$. In particular this is convenient as it assures that the Lipschitz norm of $f$ is bounded above by $||f||_{C^1}$. 
\end{Nota}

\begin{lemma} \label{comparetoAdjointaction}
The image of the derivative $\nabla_h\Xi_n( \mathrm{T}_eN^+) \subset  \mathrm{T}_e MA$ is equal to the $MA$-components of the image of the adjoint action $\Ad_n(\mathrm{T}_eN^+) \subset \mathrm{T}_eG$. 
\end{lemma}
\begin{proof}
For fixed $n$ in $N^-$ the implicit function theorem gives smooth functions
\begin{eqnarray} r_{n, 1}: N^+ &\rightarrow& N^-\\
r_{n, 2}: N^+ &\rightarrow& N^+\\
r_{n, 3}: N^+ &\rightarrow& N^-
\end{eqnarray} 
with $r_{n, 1}, r_{n, 3}$ sending the identity to $n, n^{-1}$ respectively, and $r_{n, 2}(h) = h^{-1}$ such that
\begin{eqnarray}
F(x_0) h MA &\ni& x_0 = [h, e]\\
F(x_0)hr_{n, 1}(h) MA&\ni& [h, n]\\
F(x_0)hr_{n, 1}(h) r_{n, 2}(h)MA &\ni& [h, e]\\
F(x_0)hr_{n, 1}(h)r_{n, 2}(h)r_{n, 3}(h)MA &\ni& [e, e].
\end{eqnarray}
The result is now immediate by taking derivatives in the $h$ component.
\end{proof}

We can now state  and prove the NLI property appropriate to this setting:
For sections $v_0 \ldots v_k$ of $\sigma^N$ defined on $U_1$ we define functions
$$ \phi_j(u) = a_{\tau^{(N)}}(u)\theta^{(N)}(u)$$
and 
$$BP_j(u', u) := \phi^{-1}_0(u)\phi_0(u')\phi^{-1}_j(u') \phi_j(u) \in AM.$$

\begin{lemma} \label{NLI} There exists $\epsilon_0 \in (0, 1)$, an open subset $U_0 \subset \tilde U_1$, and $N_0$ such that for any $N > N_0$ there are finitely many sections $v_0 \ldots v_{j_0}$ of $\Pp^N$ defined on $\tilde U_1$  with the following property; for any $u \in U_0$ and any unit tangent vector $w \in \mathrm{T}_eAM$  there is a unit tangent vector $z \in \mathrm{T}_u\tilde U_1$ and  $i \in \lbrace 1 \ldots r \rbrace $ such that $\langle w, \nabla_{u'}BP_i(\cdot, u)_*z\rangle \geq \epsilon_0,$ the derivative here is to be evaluated at $u$.
\end{lemma}
\begin{proof} 
Consider the base point $z_1 \in  \hat U_1$ for the rectangle $\hat R_1$ and choose  a sequence $n'_0, n'_1, \ldots,  n'_{j_0}\in N^-$  with the properties that 
\begin{itemize}
\item $n'_0 = e$,
\item $F(z_1)n'_i$ lie in the interior of $R_1$ for all $i$, and
\item $\sum_i \Ad_{n'_i} \mathrm{T}_eN^+ = \mathrm{T}_eMA$. 
\end{itemize}
This is possible using the fact that Zariski-density of $\G$ prevents the limit set $\Lambda$ from being contained in any algebraic subset of $\mathbb{R}^{n-1}$. We note that
\begin{equation} \mbox{ the third condition is stable under small perturbations of the $n'_i$.}\label{small_perturbation_condition}\end{equation}

We write $\epsilon > 0$ for the (positive) infimum
$$ \epsilon = \inf_w \sup_{X, j\geq 1} | \langle \mbox{proj}_{MA} \Ad_{n'_j}X, w \rangle |$$
where $w$ ranges over unit tangent vectors in $\mathrm{T}_eMA$, and $X$ ranges over unit tangent vectors in $\mathrm{T}_eN^+$. 

Now recall that images of the unstable piece $U_2$ under $\poi^{N_0}$ fill out $R$ densely as $N_0$ gets large. We may therefore choose $u_0 \ldots u_{j_0} \in U_2$ such that 
$$\poi^{N_0}(u_j) = [z_1, s_j] \in R_1$$
with $F([z_1, s_j]) = F(z_1) \hat n_j$ and $\hat n_j$ very close to $n'_j$. Taking $N_0$ even larger if necessary we may assume that $ F(\poi^{N_0}([u_j, s]))= F(z_1) \hat n_j(s) $ with $\hat n_j(s)$ extremely close to $n_j$ for all $s \in S_2$ in the sense of \eqref{small_perturbation_condition}; we're just contracting out the stable direction. 

Now each of these $u_j \in U_2$ defines an admissible sequence of length $N_0$ from $U_2$ to $U_1$, and therefore defines a section $\hat v_j$ of $\Pp^{N_0}$ taking $\tilde U_1$ into $\tilde U_2$. For any $N \geq N_0$ we may therefore choose an arbitrary section $\hat v $ of $\Pp^{N - N_0}$ defined on $U_2$, and pick $v_j = \hat v \circ \tilde v_j$, which is now a section of $\Pp^N$. 

Note that $P^{N - N_0}(z_1) \in \hat R_2$, We may choose elements $s_j \in S_2$ such that $P^{N - N_0}(z_1) = [u_j, s_j]$. Write $n_j := \hat n_j(s_j)$. Then $n_j$ is extremely close to $n'_j$. Furthermore $n_j$ is exactly the stable horospherical element associated to the section $v_j$ as in \eqref{horospherical_elements_for_sections}.



We claim that this choice of $v_0 \ldots v_{j_0}$ satisfies the conditions of Lemma \ref{NLI}. Pick a unit vector $w \in \mathrm{T}_eMA$. Then we may choose an index $j\geq 1$ and a unit tangent vector $X \in \mathrm{T}_eN^+$ such that $|\langle \mbox{proj}_{MA} \Ad_{n'_j} X, w \rangle |> \epsilon$. Then 
\begin{equation} \label{section5_useful_bound_1} |\langle \mbox{proj}_{MA} \Ad_{ n_j } X, w \rangle |> \epsilon /2 \end{equation} 
and 
\begin{equation} \label{section5_useful_bound_2} |\langle \mbox{proj}_{MA} \Ad_{ n_{j_0}} X, w \rangle |< \epsilon /4 \end{equation} 
since $n_j$ is very close to $n'_j$, and $n_0$ is very close to $n'_0 = e$. Identifying $\tilde U_1$ with $N^+$ as above and quoting Lemma \ref{Brin_pesin_versus_path_maps} we now have 
$$BP_j(h', h) = \Xi^{-1}_{n_0}(h) \Xi_{n_0}(h')  \Xi^{-1}_{n_j}(h')  \Xi_{n_j}(h) $$ 
where $h$ and $h'$ are the $N^+$ coordinates of $u, u'$. Note from this expression that $BP_j$ is smooth and that it has a uniform bound on it's $C^2$ norm that is {\it independent} of $N$ and of the sections $v_j$ which we chose (this is because all the $n_j$'s that arise come from the compact set $S_1$). We specialize to the case 
$$BP_j(h', e) = \Xi^{-1}_{n_0}(e) \Xi_{n_0}(h')  \Xi^{-1}_{n_j}(h')  \Xi_{n_j}(e) $$ 
and take derivatives in the $h'$ direction at $h' = e$. Lemma \ref{comparetoAdjointaction} tells that the derivative is 
$$  \mbox{proj}_{MA} ( \Ad_{ n_0 } -  \Ad_{ n_j }). $$
Applying the bounds \eqref{section5_useful_bound_1} and \eqref{section5_useful_bound_1} above we are finished, at least for $u= z_1$ (that is, for $h = e$). But now we can apply our comment on uniform bounds for the $C^2$ norms of $BP_j$ to conclude that this same inequality extends to some small neighbourhood $u \in U_0\subset \tilde U_1$ of $z_1$ as required.   
\end{proof}

The point of this statement for our purposes is is follows. For any unitary representation $(\mu_b, V_\mu)$ of $AM$, any unit vector $v  \in V_\mu$, and any $u_0 \in U_0$ we want to be sure that at least one $BP_j(u, u_0) \cdot v$ is ``rapidly oscillating''; in this context rapid oscillation should mean that that it is moving at a speed comparable to $||\mu_b||$. We make this precise as follows.

\begin{lemma} \label{actionmoveseveryone}
Let $\mu, V_\mu$ be an isotypic component of $L^2(M)$, and let $v \in V_\mu$ be a unit vector. There is a unit vector $X$ in $\mathrm{T}_eMA$ such that $|X\cdot v| \geq \delta_3 ||\mu_b||$. 
\end{lemma}

\begin{proof} We note that $||\mu_b||$ is bounded above by the length of the highest weight vector $\lambda$ associated to $\mu_b$. Conversely the eigenvalue $c(\mu_b)$ of the Casimir operator on $V_{\mu_b}$ is approximately $|\lambda|^2$. It follows (using the fact that each element of $\mathrm{T}_eMA$ acts diagonalizably on $V_\mu$) that some unit tangent vector $X \in \mathrm{T}_eMA$ satisfies
\[ |X^2 \cdot v|  \gg |\lambda|^2 / (\dim M). \] 
But then $|X\cdot v| \gg |\lambda|$ just because the operator norm of $X$ is bounded above by $||\mu_b|| \sim |\lambda|$. 
\end{proof}

\section{Dolgopyat operators in higher dimensions}

We now have all the pieces we need to prove Theorem \ref{transfer_operator_bounds}. The essential point is the construction of so-called Dolgopyat operators, which we will carry out in this Section. This is an intricate technical argument, but is by now well understood by experts, and can be adapted without conceptual challenge to our current setting. Our first step will be to fix various constants whose significance will appear gradually throughout the Section. We will then recall various standard a prior estimates. Finally we will construct our Dolgopyat operators and establish their necessary properties.

\subsection{Constants, notation,  and a priori estimates}

We recall now the constants $c_0, \kappa, \kappa_1$ from \eqref{definec_0kappakappa_1}. For simplicity we assume that $1 < \kappa < 2$ (so that $\kappa - 1 \in (0, 1))$. Let
$$U_0 \subset \tilde U_1, n_0 \in \mathbb{N},\mbox{  and  } \epsilon_0 \in (0, 1)$$ 
be the open set and constants whose existence is guaranteed by Lemma \ref{NLI}. Let 
$$ d_0 > 0 \mbox{ be the constant from Notation \ref{define_d_0}, and assume $0 < d_0 < 1$}. $$
We now choose $n_1 \in\mathbb{N}$ and length $n_1 + 1 $ cylinders $X_1, \ldots, X_{k_0}$ contained  in $U_0$ such that $\Pp^{n_1} X_k = \hat U_k$.  Recall the eigenfunction $h_\delta \in C^1(\tilde U)$ from \eqref{defineh_delta} and set 
\begin{equation} \label{define_A_0_prime} A'_0  > \frac{32(\delta + 1)}{c_0(\kappa - 1)} \max( ||\tau \circ \Pp^{-1}||_{C^1}, ||\theta\circ\Pp^{-1}||_{C^1}, ||h_\delta||_{C^1} , ||\log h_\delta||_{C^1} ),  \end{equation} 
where $||\tau \circ \Pp^{-1}||_{C^1} = \max \lbrace ||\tau_{ij} \circ \sigma^{-1}_{ij}||_{{C^1(\tilde U_j)}} : (i, j) \mbox{ is admissible} \rbrace$ and similarly for $\theta$. We fix
$$ A_0 > 4e^{A'_0},$$

which is clearly greater than $A_0'$. We choose a doubling constant $C_1 > 1$ with the property
$$\nu_k (B_{2\e} (x)) \leq C_1 \nu_k(B_{\e}(x))$$
for all $k$, all $x \in \hat U_k$, and all $\e > 0$ as in Lemma \ref{doubling_lemma}.

\begin{Nota}
For a positive real number $B$ we write $K_B = K_B(\tilde U)$ for the collection of log Lipschitz positive functions on $\tilde U$ satisfying
$$ | \nabla H(u)  | < B H(u) . $$ 
\end{Nota}

We choose $\delta_1$ small enough that the conditions of Lemma \ref{NCPlemma} hold for each cylinder $X_j$ and the constant $\delta_1$.  We let 
\begin{equation} \delta_3 \mbox{ be the constant from Lemma \ref{actionmoveseveryone} }\end{equation}
and 
\begin{eqnarray} E&\geq&\frac{2}{d_0}, \frac{2A_0}{d_0} \label{define_E} \\
 \delta_4 &\leq& \delta_1 \delta_3 \epsilon_0/7\end{eqnarray}

Let $\delta_5 < \frac{\delta_4}{20E}$. Choose
\begin{equation} \label{define_N_0} N_0 \geq \frac{\log 2E}{\log \kappa}, \frac{ \log(16E/\delta_1c_0)}{\log \kappa }, n_0 , \frac{ - \log c_0}{\log \kappa} , \frac{\log 8}{\log \kappa}, -\frac{\log \delta_4 /80Ec_0}{\log \kappa}. \end{equation}
Write $N = N_0 + n_1$ and let $v_0,  \ldots,  v_{j_0}$ be the sections of $\Pp^N$ whose existence is now guaranteed by Lemma \ref{NLI}. We choose $C_3$ large enough that
\begin{equation} \label{define_C_3} \nu_k(B_{100 \epsilon \kappa_1^{n_1} / c_0}) \leq C_3 \nu_k(B_{c_0 \delta_5 \epsilon /2}); \end{equation}
for all $k$, all $x \in \hat U_k$, and all positive $\epsilon$; precisely, choosing $C_3 = C_1^\lambda$ where $\lambda = \log_2(\frac{800\kappa_1^{n_1}}{c_0\delta_5})$ suffices. 
Choose positive constants 
\begin{eqnarray}
 \epsilon_1&\leq& \frac{\log 2}{20(1 + \delta) E},  \frac{1}{160E}, \frac{c_0 \log 2}{200 E \kappa_1^{n_1}}, \frac{\delta_4\min(1, d_0)}{25 ||BP_j||_{C^2}}, \frac{\delta_4}{25A_0^2} \\
 \epsilon_5 &\leq& \delta_1 \delta_3 \epsilon_0\epsilon_1 \label{define_eta} \\
    \eta &\leq& \frac{1}{4k_0}, \frac{c_0 \epsilon_1 \delta_5}{16k_0\kappa_1^{N_0}}, \frac{\delta_4^2 \epsilon^2_1}{8k_0}\\
    a_0 &<& 1, \delta, \frac{1}{2NA_0} \log\left(1 + \frac{\eta e^{-NA_0}}{2C_3}\right) \label{define_a_0} \\
    \epsilon_2 &\leq &\frac{\eta^2 e^{-2NA_0}}{4C_3^2}. \label{define_epsilon_2}
\end{eqnarray}

\subsection{Preparatory lemmas}
We recall now a number of well known lemmas, all of which follow from direct calculation and from expansion/contraction properties of $\Pp$.  Let $BP_j$ be the map appearing in Lemma \ref{NLI} (with respect to our current choices of $v_j$). We include some proof details not so much for novelty, but rather as a model of the arguments that are extensively throughout this section.

\begin{lemma} We have the bound
$$ |\nabla_{u'}BP_j(\cdot, u) | \leq \frac{A_0}{4}.$$
\end{lemma}
\begin{proof} This is a simple consequence of the contraction properties of $\sigma^{-1}$. Let $m\in \mathbb{N}$ and $v$ be any section of $\Pp^m$ (in the sense described just after \eqref{definec_0kappakappa_1}). We may then calculate 
\begin{eqnarray*} | \nabla (\tau^{(m)}\circ v)| \leq \sum_1^m |\nabla \tau (v\circ \Pp^{j-1})| \leq  ||\tau\circ \sigma^{-1}||_{C^1} \sum_0^{m-1}\frac{1}{c_0\kappa^{m-1-j}} \leq\frac{A_0}{32} . \end{eqnarray*} 
The point of this calculation is that the bound does {\it not} depend on $m$. (Notation has been somewhat abused here: we work on $\tilde U$, yet $\sigma$ is defined only on $\hat U \subset \tilde U$; it's OK in this instance, since $\sigma^k$ extends smoothly and naturally to $v(\tilde U)$ whenever $k \leq m$.)

A similar calculation yields $| \nabla (\theta^{(m)}\circ v) | \leq \frac{A_0}{32}$. Putting these two statements together with the definition of $BP_j$ we obtain our result. 
\end{proof}

\begin{remark}
It is useful to note that
\begin{equation}|  \tilde{\mathcal{L}}_{s, \mu} h| \leq \tilde{\mathcal{L}}_{\Re(s)} |h| \end{equation} 
pointwise for any $s \in \mathbb{C}$, any irreducible representation $\mu$, and any continuous function $h \in C(\tilde U, V_\mu)$; this follows by direct calculation. As a consequence one sees the operator norms obey
\begin{equation} || \tilde{\mathcal{L}}_{s, \mu} ||_{L^2(\nu)} \leq ||\tilde{\mathcal{L}}_{\Re(s)}||_{L^2(\nu)}. \label{real_parts_dominate} \end{equation} 
\end{remark}

\begin{lemma}[Lasota-Yorke]  \label{LY_ineq} For any real numbers $a, b$ and any non-trivial isotypic representation $\mu$ of $M$ satisfying $|\delta - a| < 1$ the following hold
\begin{itemize}
\item if $H \in K_B(U)$ for some $B > 0$ then $\tilde{\mathcal{L}}_{a}^m H \in K_{A_0\left( 1 + \frac{B}{\kappa^m}\right)}$ 
for any $m \in \mathbb{N}$, and 
\item if $H \in C^1(\tilde U, \mathbb{R})$ and $h \in C(\tilde U, V_\mu)$ and $B > 0$ satisfy 
$$ |\nabla h(u)| \leq BH(u)$$
then we have
$$|  \nabla (\tilde{\mathcal{L}}_{a + ib, \mu}^mh)(u)  | \leq A_0 \left[ ||\mu_b|| (\tilde{\mathcal{L}}_a^m|h|)(u)  + \frac{B}{\kappa^m} \tilde{\mathcal{L}}_a^mH)(u)\right]  $$
for any $m \in \mathbb{N}$. 
\end{itemize} 
\end{lemma}
\begin{proof}
Again this is simply a consequence of the product rule and the contraction properties of sections for $\Pp$. We'll sketch the argument for the first part; the second part, we claim, is very similar, though with somewhat more notational pain. We compute, using the product rule,
\begin{eqnarray*} |\nabla  \tilde{\mathcal{L}}_{a}^m H  | &\leq& \sum |\nabla e^{-a\tau^{(N)}(v(u)) - \log h_\delta(v(u)) + \log h_\delta (u) } H(v(u)) |\\ 
&\leq& \sum || -a\tau^{(N)}(v(u)) - \log h_\delta(v(u)) + \log h_\delta (u) ||_{C^1}  \tilde{\mathcal{L}}_{a}^m H  \\
&&+   \sum  e^{-a\tau^{(N)}(v(u)) - \log h_\delta(v(u)) + h_\delta (u) }| \nabla H| (v(u)) | \nabla v|(u)\\
&\leq& A_0  \tilde{\mathcal{L}}_{a}^m H +  \frac{1}{c_0 \kappa^m}  \tilde{\mathcal{L}}_{a}^m |\nabla H| 
 \end{eqnarray*} 
 where the sum is taken over all sections of $\Pp^m$; we've used some of the estimates and comments from the proof of the previous lemma here. Since $H \in K_B$ we can estimate this last in terms of $H$ to get 
 \begin{eqnarray*} |\nabla  \tilde{\mathcal{L}}_{a}^m H  | &\leq&  \left(A_0  + \frac{B}{c_0\kappa^m}\right) ( \tilde{\mathcal{L}}_{a}^m H) (u)\end{eqnarray*}
which is adequate. The proof of the second part simply involves running the same calculation with the slightly more complicated operator $\tilde{\mathcal{L}}_{s, \mu}$. 

\end{proof}

\begin{lemma} \label{triangle_ineq}
Suppose $v, w$ are vectors in some Hilbert space with $|v| \leq |w|$. Suppose that the normalised vectors $\hat v = v / |v|$ and $\hat w = w/|w|$ satisfy $|\hat v - \hat w| > \delta_4\epsilon_1$ then  $| v  + w|\leq (1 - k_0 \eta)|v| + |w|$. 
\end{lemma}
\begin{proof}This is elementary; see, for example \cite[Lemma 5.2]{OWhj} and \eqref{define_eta}.  \end{proof}
\begin{lemma} \label{h_isp_small_or_large}
Suppose that that $h \in C^1(\tilde U,V_\mu)$,  $R \in \mathbb{R}$, and $H \in K_{ER}(\tilde U)$ satisfy 
$$ |h| < H \mbox{ and } |\nabla  h(u)| \leq ER\cdot H(u).  $$
Then for any $x \in U_1$, and any section $v$ of $\Pp^{N_0}$ that is defined on $U_1$ we have either
\begin{itemize}
\item $|h \circ v| \leq \frac{3 H\circ v}{4}$  on $ B_{10\epsilon_1/R}(x), $ or
\item $|h \circ v| \geq \frac{H \circ v}{4}$  on $ B_{10\epsilon_1/R}(x). $
\end{itemize}
\end{lemma} 
\begin{proof}
We assume that the second alternative does not hold, in other words that there is a point $y \in B_{10 \epsilon_1/R}(x)$ with $|h \circ v| \leq \frac{H \circ v}{4}$ at $y$. We now want to run a calculation in Lipschitzness and the small size of the ball to conclude that the first alternative does hold. It's useful to make one preliminary calculation first, however.

Since $H \in K_{ER}$ we know that $\log H$ is $ER$-Lipschitz. Thus for any $z_1, z_2 \in B_{10 \epsilon_1/R}(x)$ we have
$$ |\log H\circ v (z_1)  - \log H\circ v(y)| \leq ERd(z_1, z_2) ||\nabla v||_{\infty} \leq \frac{20E\epsilon_1}{c_0 \kappa^{N_0}}\leq \log 2.$$
It follows that 
\begin{eqnarray} |h(v(z))| &\leq& | h(v(y)) | + \int_y^z |\nabla h \circ v|(s) ds \\
&\leq& \frac{H(v(y))}{4} +  ER \int_y^z  H \circ v (s) ds\\
&\leq&  \frac{H(v(y))}{4} +  2ER \cdot  H \circ v (z) d(y, z) \\
&\leq& \left(  \frac{1}{2} +2ER \frac{20 \epsilon_1}{R} \right)H(v(z))\\
&\leq& \frac{3}{4} H(v(z))\end{eqnarray} 
as required. This kind of Lipschitzness argument for $H$ is very helpful. \end{proof}

\subsection{Dolgopyat operators} Fix now
$$a, b \in \mathbb{R}$$
and let $(\mu, V_\mu)$ be a non-trivial isotypic component of $L^2(M)$ considered as an $M$-space. We assume that
\begin{equation} |a - \delta| < a_0.\label{a_less_than_a_0} \end{equation} 
We write 
\begin{equation} \tilde \epsilon := \epsilon_1 /||\mu_b||, \label{define_tilde_epsilon} \end{equation}
this establishes the length scale at which we wish to work, and will be of essential importance throughout this Section. We aim to prove the main spectral bounds for this choice of $a, b, \mu$.

By Vitali covering, we choose finite subsets $\lbrace x^k_r; r = 1 \ldots  r_0 = r_0(k)\rbrace $ of $X_k$ with the properties that 
$$ X_k \subset \cup_1^{r_0} B_{50 \tilde \epsilon} (x^k_r)$$
and that 
\begin{equation} \label{disjointness} B_{10 \tilde \epsilon} (x^k_{r_1}), B_{10 \tilde \epsilon} (x^k_{r_2}) \end{equation}
 are disjoint unless $r_1 = r_2$

Let $h : \tilde U \rightarrow V_\mu$ be a $C^1$ function. 

\begin{lemma}
There is a choice $j = j(k, r)$ of section from $v_0 \ldots v_{j_0}$, and a choice $y^k_r \in \hat U_1 \cap B_{5\tilde \epsilon}(x^k_r)$ such that 
$$ |   (\nabla  BP_j(\cdot, x^k_r) \cdot (y^k_r - x^k_r)) \cdot  w |  > 7\delta_4 \epsilon_1 ; $$
here we write $w$ for the unit vector in direction $\Phi_{N_0}^{-1}(v_0(x^k_r))h(v_0x^k_r)) $.
\end{lemma}

\begin{proof}
By Lemma \ref{actionmoveseveryone} the operator norm of the map $\Xi_1: \mathrm{T}_eMA \rightarrow V_\mu$ given by $Z \mapsto Z\cdot w$ is at least $\delta_3||\mu_b||$. We may therefore choose a unit vector $\tilde w \in V_\mu$ such that the dual $\tilde v := \Xi_1^* \tilde w \in T_eMA$ has norm at least $\delta_3||\mu_b||$. By Lemma \ref{NLI} we may also choose $j$ such that
$$\Xi_2 : \mathrm{T}_{x^k_r}\tilde U_1 \rightarrow \mathbb{C}, Z \mapsto \langle \nabla BP_j(\cdot, x^k_r)Z, \tilde v \rangle $$
has norm at least $\delta_3 ||\mu_b|| \epsilon_0$. In other words the dual vector
$$\hat{\hat{v}} = (\nabla BP_j(\cdot, x^k_r))^*\tilde v$$
 in $T_{x^k_r} \tilde U$ has norm at least $\delta_3 ||\mu_b|| \epsilon_0$. We now apply Lemma \ref{NCPlemma} to obtain a choice $y^k_r \in X_k \cap B_{5\tilde \epsilon}(x^k_r)$ with 
$$ |  \langle  \hat{\hat{v}},(y^k_r - x^k_r)  \rangle  | \geq 5\delta_1 \delta_3 \epsilon_0 \tilde \epsilon  ||\mu_b||$$
as required. 
\end{proof}

To each $x^k_r$ we associate a partner point $y^k_r$ as in this lemma.  For $x \in \tilde  U$ let $\psi_{x, \epsilon}$ be a bump function on $\tilde U$ taking the value $1$ on $B_{\epsilon / 2} (x)$ and $0$ outside $B_\epsilon(x)$. Without loss of generality we may assume that 
\begin{equation} ||\psi_{x, \epsilon}||_{C^1} < \frac{4}{\epsilon}. \label{bound_psi} \end{equation} 

We will think of an element 
\begin{equation} (p, \ell, r, k)\in \lbrace1, 2 \rbrace \times \lbrace1, 2 \rbrace \times \lbrace1, \ldots, r_0 \rbrace \times \lbrace1,\ldots,  k_0 \rbrace \label{think_of_plrk} \end{equation} 
 as coding choices $x_r^k$ (if $\ell = 1$) or $y^k_r$ (if $\ell = 2)$ and $v_0$ (if $p = 1$) or $v_{j(r, k)}$ (if $p = 2$). Such an element gives a function on $\tilde U$ as follows:
 
\begin{equation} \tilde \psi_{(p, \ell, r, k)} = \chi_{v_{j}(\tilde U_1)} \psi_{w, \delta_5 \tilde \epsilon} \circ \Pp^{N_0} \label{define_tilde_psi} \end{equation}
where $w= x^k_r$ if $\ell = 1$, and $w = y^k_r$ if $\ell = 2$, and $j = 0$ if $p = 1$ and $j = j(k, r)$ if $p = 2$. (Strictly speaking this is somewhat imprecise, as $\Pp$ us defined only on $U$, not on $\tilde U$; the restriction to  $v_j(\tilde U_1)$ allows us to make sense of this in unambiguous fashion.) Unpicking this, we merely mean that $\tilde \psi_{(p, \ell, r, k)}$ is a bump function centered near $v_0(x^k_r), v_0(y^k_r), v_{j(r, k)}(x^k_r)$ or $ v_{j(r, k)}(y^k_r)$ as appropriate. 
 
Given a subset $J \subset \lbrace 1, 2 \rbrace \times \lbrace1, 2 \rbrace \times \lbrace1, \ldots, r_0 \rbrace \times \lbrace1,\ldots,  k_0 \rbrace$ we choose the function 
\begin{equation} \beta_J = 1 - \eta \sum_J  \tilde \psi_{(p, \ell, r, k)}. \label{define_beta_J} \end{equation}
We say that the subset $J$ is full if, for every $(r, k)$ there exists exactly one pair $p, \ell$ with $(p, \ell, r, k) \in J$. We write $\mathcal{F}$ for the collection of full subsets. 

For a full subset $J$ we define the Dolgopyat operator
\begin{equation} \mathcal{M}_{J, a}H := \tilde{\mathcal{L}}^N_{a, 0}(H\beta_J). \label{define_Dolgopyat_operator} \end{equation} 

\subsection{Bounding transfer operators by Dolgopyat operators}
The two points of Dolgopyat operators are that they are contracting and that they dominate their partner transfer operators. The contraction property is comparatively straightforward:
\begin{theorem} \label{Dolgopyat_operators_contract}
If $H \in K_{E||\mu_b||}$ and $J$ is full then 
\begin{itemize}
\item $ \mathcal{M}_{J, a}H  \in K_{E||\mu_b||}$, and 
\item $|| \mathcal{M}_{J, a}H||_{L^2} \leq (1 - \epsilon_2)||H||_{L^2}$. 
\end{itemize} 
\end{theorem} 
\begin{proof} This goes back at least as far as \cite{Do}. We recount these now-standard arguments for the convenience of the reader. 
For the first part we calculate:
\begin{eqnarray*}
|\nabla \beta_J H| &\leq& |\nabla H| + H |\nabla \beta_J |\\
&\leq& E||\mu_b||H + H \sum_{(p, l, r, k) \in J}|\nabla \tilde \psi_{(p, l, r, k)}|\\
&\leq& \left( E||\mu_b|| + \frac{4k_0\eta }{\delta_5 \tilde \e} \sup |\nabla \Pp^{N_0}| \right) H;
\end{eqnarray*}
the last term requires some explanation: for any $x$ at most $k_0$ quartets $(p, l, r, k) \in J$ have $\supp \tilde \psi_{(p, l, r, k)} \ni x$ by \eqref{disjointness}; each of those satisfy
$$|\nabla \tilde \psi_{(p, l, r, k)}| \leq \frac{4 }{\delta_5 \tilde \e} \sup |\nabla \Pp^{N_0}|$$
using \eqref{bound_psi} and \eqref{define_tilde_psi}; the $\eta$ term then appears from \eqref{define_beta_J}. Applying our choices of constants \eqref{define_eta} we then see that 
$$| \nabla \beta_J H| \leq \left(    E||\mu_b|| + \frac{4k_0\eta \kappa_1^{N_0} }{\delta_5 c_0 \tilde \e} \right) H\leq ||\mu_b||(E+1)H$$
using \eqref{define_tilde_epsilon}. Thus $\beta_J H \in K_{2E||\mu_b||}$ by Lemma \ref{LY_ineq}, and 
$$M_{J, a} H = \tilde{\mathcal{L}}_{a}^{N}(\beta_J H) \in  K_{A_0(1 + 2E||\mu_b||/\kappa^N)} \subset  K_{A_0(1 + ||\mu_b||)} \subset K_{E||\mu_b||}  $$
as required.

The second part is somewhat more involved. We seek to estimate $||\mathcal{M}_{J, a} H||_{L^2(\nu)}$ for $J$ a full subset. We start by considering the case $a = \delta $ and will later use continuity in $a$ to claim that that suffices. 

A direct calculation in Cauchy Schwartz gives
\begin{equation}  (\mathcal{M}_{J, \delta} H )^2 \leq ( \tilde{\mathcal{L}}^N_\delta (H^2))( \tilde{\mathcal{L}}^N_\delta (\beta_J^2)).  \label{CS_says} \end{equation}
We recall that $\tilde{\mathcal{L}}_{\delta} 1 = 1$ and that $\tilde{\mathcal{L}}_\delta$ is linear and sends positive functions to positive functions. Thus $ \tilde{\mathcal{L}}^N_\delta (\beta_J^2) \leq 1$, and we can hope to be done if we can find a large set where it is bounded strictly away from one. That's our first task. 

Let $\tilde S_k$ be the union of the sets
$$ \lbrace x^k_r: v_j(x^k_r) \in J \mbox{ for some } 1 \leq j \leq j_0 \rbrace$$
and
$$  \lbrace y^k_r: v_j(y^k_r) \in J \mbox{ for some } 1 \leq j \leq j_0 \rbrace$$
both understood in the sense of \eqref{think_of_plrk}. Now, $\Pp^{n_1}(X_k) = \hat U_k$. Let $S_k = \Pp^{n_1}(\tilde S_k)$. Consider the neighborhood
\begin{equation} \hat S_k := B_{c_0 \delta_5 \tilde \e \kappa^{n_1}/2}(S_k), \label{define_hat_S_k}\end{equation} 
which we think of as a subset of $\tilde U_k$ For any $y \in \hat S_k$ there is at least one section $\tilde v$ of $\Pp^N$ such that $\beta_J^2(\tilde v(y)) \leq \beta_J(\tilde v(y)) \leq 1 - \eta$;  this uses \eqref{define_hat_S_k}, \eqref{define_tilde_psi}, and \eqref{definec_0kappakappa_1} . Then 
\begin{eqnarray*} 
\tilde{\mathcal{L}}_\delta^{N}\beta^2_J(y) &=& \sum_{\mbox{ sections $v$ of $\Pp^N$}}e^{-\delta\tau^{(N)}(v(y))}\frac{\beta^2_J(v(y)) h_\delta(v(y))}{h_\delta(y)}\\
&\leq& \left( \sum_{\mbox{ sections $v$ of $\Pp^N$}}e^{-\delta\tau^{(N)}(v(y))}\frac{ h_\delta(v(y))}{h_\delta(y)} \right)  - \eta e^{-\delta\tau^{(N)}(\tilde v(y))}\frac{ h_\delta(\tilde v(y))}{h_\delta(y)} \\
&=& 1 - \eta e^{-\delta\tau^{(N)}(\tilde v(y))}\frac{ h_\delta(\tilde v(y))}{h_\delta(y)} \\
&\leq& 1 - \eta e^{-NA_0}
\end{eqnarray*}
 using \eqref{define_A_0_prime}.  At this stage we conclude that
\begin{equation} \tilde{\mathcal{L}}_\delta^{N}\beta^2_J \leq 1 - \eta e^{-NA_0} \mbox{ on } \hat S_k. \label{bound_away_from_one}  \end{equation} 

Our second observation is that 
\begin{equation} \mbox{The support of $\nu_k$ is contained in $ B_{100  \tilde \e \kappa_1^{n_1}/c_0}(S_k)$}; \label{hat_S_k_not_too_small} \end{equation}
This follows simply because $X_k$ is contained in $B_{100 \tilde \e}(\tilde S_k)$ and because of our Lipschitz bounds \eqref{definec_0kappakappa_1} on $\Pp^{n_1}$. It should be taken to mean that the set on which \eqref{bound_away_from_one} holds is not too small.

We write 
$$ \tilde H := \tilde{\mathcal{L}}_\delta^N (H^2)$$
throughout the rest of this proof. 

The last piece of our puzzle is to establish that the integral of $\tilde H$ over $\hat S_k$ is not too small. This relies in an essential way on the doubling property of $\nu_k$, on \eqref{hat_S_k_not_too_small}, and on the regularity of $\tilde H$. 

\noindent \textbf{Claim:} For any $z \in S_k$ we have
\begin{equation}  \sup_{w \in B_{100 \tilde \e \kappa_1^{n_1}/c_0}(z) }\tilde H(w) \leq  2 \inf_{w \in B_{100 \tilde \e \kappa_1^{n_1}/c_0}(z) }\tilde H(w). \label{factor_2} \end{equation}
\noindent \textbf{Proof of claim:} Note that  $H \in K_{E||\mu_b||} \Rightarrow H^2 \in K_{2E||\mu_b||}$, and a calculation using Lemma \ref{LY_ineq} shows $\tilde H \in K_{E||\mu_b||}$; use \eqref{define_N_0} and \eqref{define_E}. We therefore have, for any $x, y \in  B_{100 \tilde \e \kappa_1^{n_1}/c_0}(z) $
\begin{eqnarray*}
\log \tilde H(x) - \log \tilde H(y) &\leq& 200  E||\mu_b|| \tilde \e \kappa_1^{n_1} / c_0\\
&\leq& 200E\epsilon_1 \kappa_1^{n_1} /c_0\\
&\leq& \log 2,
\end{eqnarray*}
which establishes the claim.

We are now ready to start stringing these observations together. 

\noindent \textbf{Claim:} We have the inequality 
\begin{equation}  \int \tilde  H d\nu_k \leq  2 C_3 \int_{\hat S_k} \tilde H\label{bound_int_over_hat_S_k} \end{equation} 

\noindent \textbf{Proof of claim:} Simply calculate:
\begin{eqnarray*}   \int \tilde  H d\nu_k  &\leq& \sum_{z \in S_k  } \int_{B_{100 \tilde \e \kappa_1^{n_1}/c_0}(z)} \tilde H    d\nu_k \mbox{ by \eqref{hat_S_k_not_too_small}} \\
&\leq&  \sum_{z \in S_k  } \int_{B_{100 \tilde \e \kappa_1^{n_1}/c_0}(z)} d\nu_k  \left( \sup_{w \in B_{100 \tilde \e \kappa_1^{n_1}/c_0}(z) }\tilde H(w)\right) \\
&\leq&2 \sum_{z \in S_k  } \int_{B_{100 \tilde \e \kappa_1^{n_1}/c_0}(z)} d\nu_k  \left(  \inf_{w \in B_{100 \tilde \e \kappa_1^{n_1}/c_0}(z) }\tilde H(w)\right) \mbox{ by \eqref{factor_2}} \\
&\leq&2C_3  \sum_{z \in S_k  } \int_{B_{ c_0 \delta_5 \tilde \e \kappa^{n_1}/2}(z)} d\nu_k  \left(  \inf_{w \in B_{100 \tilde \e \kappa_1^{n_1}/c_0}(z) }\tilde H(w)\right) \mbox{ by \eqref{define_C_3}} \\
&\leq&2C_3  \sum_{z \in S_k  } \int_{B_{ c_0 \delta_5 \tilde \e \kappa^{n_1}/2}(z)} \tilde H d\nu_k \\
&\leq& 2C_3  \int_{ \hat S_k} \tilde H d\nu_k \mbox{ by \eqref{define_hat_S_k}}
  \end{eqnarray*} 
  establishing the claim. 
  
 Using this claim we have 
 \begin{eqnarray*} && \int_{\tilde U_k} \tilde{\mathcal{L}}_\delta^N (H^2)d\nu_k - \int_{\tilde U_k} ( \mathcal{M}_{J, \delta} H)^2 d\nu_k  \\
 &\geq&  \int_{\tilde U_k} \tilde{\mathcal{L}}_\delta^N (H^2)  -  ( \tilde{\mathcal{L}}^N_\delta (H^2))( \tilde{\mathcal{L}}^N_\delta (\beta_J^2)) d\nu_k  \mbox{ by \eqref{CS_says}} \\
 &\geq&    \int_{\tilde U_k} \tilde{\mathcal{L}}_\delta^N (H^2) ( 1 -  \tilde{\mathcal{L}}^N_\delta (\beta_J^2)) d\nu_k \\
  &\geq& \int_{\hat U_k} \tilde{\mathcal{L}}_\delta^N (H^2)  d\nu_k \mbox{ since $\beta_J \leq 1$ and $ \tilde{\mathcal{L}}^N_\delta (\beta_J^2)  \leq 1$. } \\ 
 &\geq&  \eta e^{-NA_0}  \int_{\hat S_k} \tilde{\mathcal{L}}_\delta^N (H^2)  d\nu_k \mbox{ by \eqref{bound_away_from_one}}  \\ 
 &\geq& \frac{\eta e^{-NA_0}}{2C_3}  \int_{\tilde U_k} \tilde{\mathcal{L}}_\delta^N (H^2)  d\nu_k  \mbox{ by \eqref{bound_int_over_hat_S_k}}
  \end{eqnarray*} 
  summing over $k$ and using that $\tilde{\mathcal{L}}_\delta$ preserves $\nu$ have 
  \begin{eqnarray*} ||H||_{L^2(\nu)}^2 -  ||\mathcal{M}_{J, \delta} H||_{L^2(\nu)}^2&\geq& \frac{\eta e^{-NA_0}}{2C_3} ||H||_{L^2(\nu)}^2   \end{eqnarray*} 
  Finally by \eqref{a_less_than_a_0} the definitions \eqref{define_Dolgopyat_operator} and \eqref{define_transfer_operator} give  
  \begin{eqnarray}  ||\mathcal{M}_{J, a} H||^2 &\leq& e^{2a_0 N A_0}   ||\mathcal{M}_{J, \delta} H||^2\\
  &\leq&  e^{2a_0  N A_0}\left(1 -  \frac{\eta e^{-NA_0}}{2C_3}\right)  ||H||^2    \\
  &\leq& (1 - \epsilon_2) ||H||^2 \end{eqnarray} 
  by \eqref{define_a_0} and \eqref{define_epsilon_2}
 \end{proof}

The other requirement for Dolgopyat operators is that they dominate the associated transfer operators as follows.  This is more involved again, but there are now well known techniques to complete the task. We outline the main steps below. 

\begin{theorem} \label{Dolgopyat_operators_dominate}
For every $h \in C^1(\tilde U, V)$ and every $H \in K_{E||\mu_b||}$ satisfying $ |h| \leq H $ and 
$$  | \nabla h(u)   | \leq E||\mu_b ||H(u)$$
there is a full $J$ such that 
$$ |\tilde{\mathcal{L}}^N_{s, \mu} h  |\leq \mathcal{M}_{J, a} H $$
and
$$ | \nabla \tilde{\mathcal{L}}^N_{s, \mu} h(u) |\leq  E||\mu_b|| \mathcal{M}_{J, a} H(u).$$
\end{theorem}
\begin{proof}[Sketch proof of Theorem \ref{transfer_operator_bounds}]
The connection of Theorems \ref{Dolgopyat_operators_contract} and \ref{Dolgopyat_operators_dominate} to Theorem \ref{transfer_operator_bounds} is as follows. Suppose we are given a function $h \in C^1(\tilde U, V_\mu)$. For simplicity we consider the case $||h||_{C^1} = 1$. By Theorem \ref{Dolgopyat_operators_dominate} we can inductively choose a sequences of functions $H_n \in C^1(\tilde U, \mathbb{R})$ with
\begin{itemize}
\item{ $H_0 = 1$ is constant}
\item{ for each $k$ we choose $H_{k+1} = \mathcal{M}_{J_k, a} H_k \in K_{E||\mu_b||}$ for some full subset $J_k$. Moreover we also require}
\item  $|\tilde{\mathcal{L}}^{Nk}_{s, \mu} h  |\leq H_k$ and  $| \nabla \tilde{\mathcal{L}}^{Nk}_{s, \mu} h(u) |\leq  E||\mu_b|| H_k(u)$, which is possible by Theorem \ref{Dolgopyat_operators_dominate}.
\end{itemize}
We therefore have that
$$|| \tilde{\mathcal{L}}^{Nk}_{s, \mu} h||_{L^2(\nu)} \leq ||H_k||_{L^2(\nu)}$$
for every $k$, and so (by Theorem \ref{Dolgopyat_operators_contract}) that  
$$|| \nabla \tilde{\mathcal{L}}^{Nk}_{s, \mu} h||_{L^2(\nu)} \leq (1 - \epsilon_2)^k$$
for all $k \in \mathbb{N}$. From here it is relatively straightforward to deduce Theorem \ref{transfer_operator_bounds}; for arbitrary $n  = Nk + r$ with $0 \leq r <N$ we have (using \eqref{real_parts_dominate})
\begin{equation} ||  \tilde{\mathcal{L}}^{n}_{s, \mu} h||_{L^2(\nu)} \leq ||  \tilde{\mathcal{L}}^{Nk}_{s, \mu} h||_{L^2(\nu)}  \cdot ||  \tilde{\mathcal{L}}_{s, \mu} ||^r_{L^2(\nu)} \leq (1 - \epsilon_2)^k(1 +  ||  \tilde{\mathcal{L}}_{\Re(s)} ||_{L^2(\nu)})^N.\end{equation}
At this stage we are done, since $N$ is fixed, and since the operator norm $|| \tilde{\mathcal{L}}_{\Re(s)} ||_{L^2(\nu)}$ is bounded uniformly when $\Re(s)$ is varies over bounded intervals, which we have since we assumed $|a - \delta| < a_0$.
\end{proof}

Exactly as in \cite{OWhj}, Theorem \ref{Dolgopyat_operators_dominate} will follow from the Lemma:
 \begin{lemma} \label{key_lemma}
Define
 $$\Delta_1 := \frac{| e^{-a\tau_{N_0}(v_0(x))} \Phi^{-1}_{N_0}(v_{0}(x))h(v_{0}(x)) +e^{-a\tau_{N_0}(v_j(x))} \Phi^{-1}_{N_0}(v_{j}(x))h(v_{j}(x)) |    }{ ( 1 - k_0 \eta) e^{-a\tau_{N_0}(v_0(x))} H(v_{0}(x)) +e^{-a\tau_{N_0}(v_j(x))} H(v_{j}(x)) } $$
 and 
  $$\Delta_2 := \frac{| e^{-a\tau_{N_0}(v_0(x))} \Phi^{-1}_{N_0}(v_{0}(x))h(v_{0}(x)) +e^{-a\tau_{N_0}(v_j(x))} \Phi^{-1}_{N_0}(v_{j}(x))h(v_{j}(x)) |    }{  e^{-a\tau_{N_0}(v_0(x))} H(v_{0}(x)) +( 1 - k_0 \eta) e^{-a\tau_{N_0}(v_j(x))} H(v_{j}(x)) } $$
 Then at least one of $\Delta_1$, $\Delta_2$ is less than or equal to one on at least one of $B_{2\delta_5\tilde \epsilon}(x^k_r)$,  $B_{2\delta_5\tilde \epsilon}(y^k_r)$
 \end{lemma}
 This should be read as follows: the numerator represents two summand terms from the definition of the transfer operator. The denominator (almost) represents two summand terms from the definition of the Dolgopyat operator, with the $(1 - \eta)$ standing in for $\beta_J$. The Lemma says that the Dolgopyat operator is larger at least somewhere.

 \begin{proof} If the first alternative of Lemma \ref{h_isp_small_or_large} holds anywhere on $\B_{10 \tilde \epsilon}(x^k_r)$ for either $v_0$ or $v_j$ then we are done; each of the terms on the top is then smaller than it's counterpart on the bottom, and since all the terms in the denominator are positive reals there is no chance of cancellation there complicating matters. So assume that  
 \begin{equation} \mbox{ the second alternative of Lemma \ref{h_isp_small_or_large} holds for both $v_0$, and  $v_1$ on $\B_{10 \tilde \epsilon}(x^k_r)$ }. \label{2nd_alternative_holds} \end{equation} 
 
 We fix some notation: let $j= j(r, k)$, $w_1(x) =\frac{  h(v_0(x))}{| h(v_0(x))|} \in V,  w_2(x)= \frac{  h(v_j(x))}{| h(v_j(x))|}, \phi_1(x) = \Phi^{(N_0)}(v_0(x)), \phi_2(x) = \Phi^{(N_0)}(v_j(x))$.    We should think of $\phi_j$ as rapidly oscillating, and $w_j$ as slowly oscillating. We now make that second point precise.
 
\noindent  \textbf{Claim:} The vector $w_j$ is $\frac{\delta_4 ||\mu_b||}{20}$-Lipschitz for $j = 1, 2$. 

\noindent  This claim is a consequence (and the purpose of) \eqref{2nd_alternative_holds}. We note that 
 \begin{eqnarray} |\nabla w_i(x) | &\leq& \frac{|\nabla h(v_i(x))| \cdot |\nabla v_j |}{|h(v_j(x))|}\\
 & \leq& \frac{E||\mu_b||H(v_j(x))}{|h(v_j(x))|c_0 \kappa^N}\mbox{ by hypotheses of Theorem \ref{Dolgopyat_operators_dominate}}\\
 & \leq& \frac{4E||\mu_b||H(v_j(x))}{H(v_j(x))c_0 \kappa^N} \mbox{ by \eqref{2nd_alternative_holds}}\\
 &\leq& \frac{ 4E||\mu_b||}{c_0 \kappa^N}\\ 
 &\leq& \frac{\delta_4 ||\mu_b||}{20} \mbox{by \eqref{define_N_0}}
  \end{eqnarray}
  which verifies the claim.

Some heuristics: the idea is to show that the two summands in the numerator are a long way from parallel near at least one of $x^k_r, y^k_r$. In other words we want to show that at least one of 
\begin{equation}  | \phi^{-1}_1(x^k_r)w_1(x^k_r) - \phi^{-1}_2(x^k_r) w_2(x^k_r)| \label{small_first_term} \end{equation}
and
\begin{equation}  | \phi^{-1}_1(y^k_r)w_1(y^k_r) - \phi^{-1}_2(y^k_r) w_2(y^k_r)|   \label{small_second_term}\end{equation} 
is reasonably large; we'll then quote Lemma \ref{triangle_ineq}.  More precisely we shall assume that \eqref{small_first_term} is small and use that fact to conclude that \eqref{small_second_term} is large. 

Now try to make the heuristics precise. For the current calculation we shall write $x = x^k_r$ and $y = y^k_r$ to ease notation. Compute, assuming that \eqref{small_first_term} is at most $2\delta_4\epsilon_1$, 
\begin{eqnarray*} 
| \phi^{-1}_1(y)w_1(y) - \phi^{-1}_2(y) w_2(y)|  &\geq& | \phi_2(y)  \phi^{-1}_1(y)w_1(y) - w_2(y)|  \\
&\geq&  | \phi_2(y)  \phi^{-1}_1(y)w_1(x) - w_2(x)| \\  && - |w_1(y) - w_1(x)| - |w_2(y) - w_2(x)| \\
&\geq&  | w_1(x) - \phi_1(y) \phi^{-1}_2(y)  \phi_2(x)  \phi^{-1}_1(x) w_1(x)| \\  && - |w_1(y) - w_1(x)| - |w_2(y) - w_2(x)| \\ &&-  | \phi^{-1}_1(x)w_1(x) - \phi^{-1}_2(x) w_2(x)| \\
&\geq&  | \phi^{-1}_1(x) w_1(x) - BP_j(y, x) \phi^{-1}_1(x) w_1(x)| \\  && - |w_1(y) - w_1(x)| - |w_2(y) - w_2(x)| \\ &&-  | \phi^{-1}_1(x)w_1(x) - \phi^{-1}_2(x) w_2(x)|
\end{eqnarray*}
The claim above on Lipschitzness of $w_j$, together with the fact $d(x, y) \leq 5 \tilde \epsilon_1$, implies $|w_j(x) - w_j(y)| < \delta_4 \epsilon_1 / 4$. Thus
\begin{equation} | \phi^{-1}_1(y)w_1(y) - \phi^{-1}_2(y) w_2(y)|   \geq  | \phi^{-1}_1(x) w_1(x) - BP_j(y, x) \phi^{-1}_1(x) w_1(x)|  - 3\delta_4 \epsilon_1,\end{equation}
since \eqref{small_first_term} was asssumed to be less than $2\delta_4\epsilon_1$.  We write $\tilde w_1 =  \phi^{-1}_1(x^k_r) w_1(x^k_r)$. Our choice of $y^k_r$ was exactly supposed to ensure that the $BP_j$ term is reasonably large. Let
$$Y = \nabla BP_j(\cdot, x^k_r)(y^k_r -x^{k}_r) \in \mathrm{T}_eMA.$$ 
Then 
\begin{eqnarray*}
| \tilde w_1 - BP_j(y^k_r, x^k_r) \tilde w_1| &\geq& | \tilde w_1 - \exp(Y) \tilde w_1|  - | BP_j(y^k_r, x^k_r)\tilde w_1 - \exp(Y)\tilde w_1| \\
 &\geq& | \tilde w_1 - \exp(Y) \tilde w_1|  - ||\mu_b||d(BP_j(y^k_r, x^k_r), \exp(Y))\\
&\geq& |Y \cdot \tilde w_1| -  ||\mu_b||^2|Y|^2 -   ||\mu_b||d(BP_j(y^k_r, x^k_r), \exp(Y))\\
&\geq&  |Y \cdot \tilde w_1| -  ||\mu_b||^2|Y|^2 -   ||\mu_b|| \cdot  ||BP_j||_{C^2} (d(x^k_r, y^k_r))^2\\
&\geq&  7\delta_4\epsilon_1 - \delta_4\epsilon_1 -   \delta_4\epsilon_1\\
&\geq& 5\delta_4 \epsilon_1.
\end{eqnarray*}
We therefore conclude that 
\[  | \phi^{-1}_1(y^k_r)w_1(y^k_r) - \phi^{-1}_2(y^k_r) w_2(y^k_r)| \geq 2\delta_4\epsilon_1 \] 
as expected. At this stage we should be convinced that 
$$\max \lbrace  | \phi^{-1}_1(y^k_r)w_1(y^k_r) - \phi^{-1}_2(y^k_r) w_2(y^k_r)|,  | \phi^{-1}_1(x^k_r)w_1(x^k_r) - \phi^{-1}_2(x^k_r) w_2(x^k_r)|  \rbrace  > 2 \delta_4 \epsilon_1.$$
For simplicity we assume that the first term is the larger, and claim that the other possibility is similar. 

\noindent \textbf{Claim:} Lipschitzness and our choice of small $\delta_5$ imply that 
$$| \phi^{-1}_1(y)w_1(y) - \phi^{-1}_1(y^k_r) w_1(y^k_r)| < \delta_4 \epsilon_1 / 2 $$
and
$$| \phi^{-1}_2(y)w_2(y) - \phi^{-1}_2(y^k_r) w_2(y^k_r)|< \delta_4 \epsilon_1 / 2 $$
on $y\in B_{2\delta_5\tilde \epsilon}(y^k_r)$.

From this claim it is easy to see that 
\[  | \phi^{-1}_1(y)w_1(y) - \phi^{-1}_2(y) w_2(y)| > \delta_4 \epsilon_1 \] 
for each $y\in B_{2\delta_5\tilde \epsilon}(y^k_r)$. At this stage we can apply Lemma \ref{triangle_ineq} to see that: at every $y\in B_{2\delta_5\tilde \epsilon}(y^k_r)$ at least one of $\Delta_1, \Delta_2$ is less than or equal to one (the choice depends on whether the first or second term in the numerator of $\Delta_i$ is larger). At this stage we are very nearly done; our only worry is that the choice of $\Delta_1$ versus $\Delta_2$ might be different for different points $y$. Using the estimate $H \in K_{E||\mu_b||}$ and the small diameter of $B_{2\delta_5\tilde \epsilon}(y^k_r)$ one shows that 
$$ \frac{\sup_{y} \lbrace e^{-a\tau^{(N)}(v_j(y))}H(v_j(y)) \rbrace }{\inf_{y} \lbrace e^{-a\tau^{(N)}(v_j(y))}H(v_j(y)) \rbrace } \leq 2$$
for any $j$ (the $\inf$ and $\sup$ are taken over $y\in B_{2\delta_5\tilde \epsilon}(y^k_r))$. This, together with our choice of constant $\eta$ provides enough wriggle room to absorb the ambiguity, allowing us to fix the choice of $\Delta_1, \Delta_2$ to be that given by $y = y^k_r$.
\end{proof}
\begin{proof}[Proof of Lemma \ref{key_lemma} implies Theorem \ref{Dolgopyat_operators_dominate}] Fix $h, H$ as in the Theorem, and try to construct an appropriate full subset $J \subset \lbrace 1, 2 \rbrace \times \lbrace1, 2 \rbrace \times \lbrace1, \ldots, r_0 \rbrace \times \lbrace1,\ldots,  k_0 \rbrace$. We use Lemma \ref{key_lemma}. For each $r$ and $k$:
\begin{itemize}
\item if $\Delta_1 \leq 1$ on $B_{2\delta_5\tilde \epsilon}(x^k_r)$ include $(1, 1, r, k)$ in $J$;
\item else if $\Delta_1 \leq 1$ on $B_{2\delta_5\tilde \epsilon}(y^k_r)$ include $(1, 1, r, k)$ in $J$;
\item else if $\Delta_2 \leq 1$ on $B_{2\delta_5\tilde \epsilon}(x^k_r)$ include $(1, 1, r, k)$ in $J$;
\item else if $\Delta_2 \leq 1$ on $B_{2\delta_5\tilde \epsilon}(y^k_r)$ include $(1, 1, r, k)$ in $J$.
\end{itemize}
The key lemma says that at least one of these conditions must hold, so the subset $J$ so obtained is full. The first required inequality now follows by unwrapping the condition on $\Delta_i$. The second required condition is a consequence of Lasota Yorke. This completes the sketch proof. 
\end{proof}

\section{Matrix coefficients and the measure of maximal entropy} \label{matrix_coefficient_section}

Exponential mixing for the measure of maximum entropy is a very natural question from a dynamical viewpoint. For some applications, however, it is convenient to convert this into a statement about the decay of matrix coefficients for the $G$-representation $L^2(\Gamma \ba G)$. This translation was understood by Roblin in his thesis \cite{Ro}, and effectivized by Oh and Shah \cite{OS}. For the reader's convenience we describe the necessary arguments here. Our treatment will be very similar to that in \cite{OW}. Throughout the section we assume that

\begin{center} $\G$ is convex cocompact and Zariski dense.\end{center}

\begin{proposition} Suppose that the $A$ action on $(\G \ba G, m^{\BMS})$ is exponentially mixing in the  sense of Theorem \ref{main1}. Then we have exponential decay of matrix coefficients: there exists $\eta > 0$ such that for any compactly supported functions $\phi, \psi \in C^{\infty}(\Omega)$ we have 
\[ \int_{\Gamma \ba G} \phi(g)\psi(ga_t) dg = C(\phi, \psi) e^{\delta(\G) - n + 1} + O(e^{-\eta t}).\] 
The constant $C(\phi, \psi)$ is given explicitly in terms of the Burger-Roblin measures on $\G \ba G$, and the implied constant may be expressed in terms of appropriate $C^k$ norms for $\phi, \psi$ and in terms of their supports.
\end{proposition}
\begin{remark}
It should be noted that the result would certainly be false (at least when $\delta < (n-1)/2$) if we dropped the condition of compact support for $\phi, \psi$. This is because there are functions living on the flares whose decay rate is only $O(e^{(n - 1)/2})$.  This is the primary difficulty in using representation theory to study groups with small critical exponent.
\end{remark} 

\subsection{A comment on H\"older functions} \label{comment_on_holderness_for_haar} We will state and prove our results here only in the case that $\phi, \psi$ are smooth functions. In fact, however, H\"older regularity will be sufficient for the arguments below. This is useful in the deduction of Corollary \ref{main2_coro}. The naive path to Corollary \ref{main2_coro} is to first approximate H\"older functions $\phi$ and $\psi$ by smooth functions $\phi_\e, \psi_\e$, then apply Theorem \ref{main2} and set $\e = e^-{\alpha t}$. Unfortunately it's not so clear that this works; the error term in passing from $\phi$ to $\phi_\e$ is $O(\epsilon^\theta)$, which will likely swamp the main term. Instead one should first deduce Corollary \ref{main1_coro}, then run this entire section for H\"older functions, and so reach Corollary \ref{main2_coro} by a longer but safer path. 

\subsection{A brief excursion into Patterson-Sullivan theory} We start by recalling the required aspects of Patterson-Sullivan theory. A $\G$-invariant conformal density of dimension $\delta$ is a family of probability measures $\lbrace \mu_x : x \in \partial (\mathbb{H}^n)\rbrace $ on $\partial \mathbb{H}^n$ which are mutually absolutely continuous and satisfy
\begin{itemize}
\item $\frac{d\mu_x}{d\mu_y} (\xi) = e^\beta_\xi(x, y)$, and 
\item $\gamma_* \mu_x = \mu_{\gamma x}$. 
\end{itemize}
It is a theorem of Sullivan \cite{Su} that if $\G$ is convex cocompact then there is a unique $\G$-invariant conformal density $\mu^{\PS}$ of critical dimension $\delta = \delta(\G)$. There is also always a $G$-invariant density $m$ of dimension $(n - 1)$ obtained by taking $m_o$ to be the $K$-invariant measure on $\partial (\mathbb{H}^n)$.

We denote the unstable and stable horospherical subgroups for $A$ by $N^+$ and $N^-$ respectively. Recall the visual maps from $G$ to the boundary $\partial(\mathbb{H}^n)$ defined by $g^+:= \lim_{t \rightarrow +\infty} ga_t$ and $g^+:= \lim_{t \rightarrow +\infty} ga_{-t}$. For any $g \in G$ the visual map provides an identification of the horospheres $gN^\pm  $ with the boundary $\partial (\mathbb{H}^n \setminus \lbrace g^\mp\rbrace )$; note the switch from $\pm$ to $\mp$. This allows us to build measures on horospheres as follows: for $\mu$ a conformal density of dimension $\delta_\mu$ set 

\begin{align*} d\tilde \mu_{gN^\pm}(n)&=e^{\delta_\mu \beta_{(gn)^\pm}(o, gn)} d\mu_o(gn^\pm). \end{align*}
 Note that these are not finite measures in general. The most important densities for us will be the measures coming from the Patterson-Sullivan density $d\tilde \mu^{\PS}_{gN^\pm}$ and the measures coming from the $K$ invariant density $d\tilde \mu^{\Leb}_{gN^\pm} := d\tilde m_{gN^\pm}$. Direct calculation shows that $ d\tilde \mu^{\Leb}_{gN^\pm}(n)$ is indeed Lebesgue measure on the appropriate horosphere.

\subsection{Generalized Bowen-Margulis-Sullivan measures on $\G \ba G$}

From measures on horospheres there is a procedure to construct $\G$-invariant measures on $G$. For any $g_0 \in G$ we consider the coordinates
\[ N^-\times N^+ \times A \times M \rightarrow G\] 
given by multiplication of group elements
\[ (n^-, n^+, a, m)\rightarrow g_0n^-n^+am.\]
For a pair $\mu, \nu$ of $\G$-invariant conformal densities of dimensions $\delta_\mu, \delta_\nu$ we can therefore define the generalized BMS measure on $G$, which takes a compactly supported function $\phi$ to
 \begin{equation} \tilde m^{\mu, \nu} (\phi) := \int_{g_0N^-} \int_{nN^+}  \int_{\mathbb{R}} \int_{M} \phi(ha_tm)  dm  dt d\nu_{nN^+} (h) d\mu_{g_0N^-}(n).\end{equation}  
This is independent of choice of $g_0$, and is $\G$ invariant on $G$,  so descends to a measure $m^{\mu, \nu}$ on $\G \ba G$. The measure is $A$ invariant if and only if $\delta_\nu = \delta_\mu$.  The BMS measure can be recast as 
\begin{equation} \tilde m^{\mu, \nu} (\phi) := \int_{g_0N^+} \int_{hN^-}  \int_{\mathbb{R}} \int_{M} \phi(na_tm)  dm  dt d\mu_{hN^-} (n) d\nu_{g_0N^+}(h).\label{BMS_and_PS} \end{equation}

This construction recovers some measures that we already care about: the measure of maximal entropy, for instance, is given by taking $\mu = \nu = \mu^{\PS}$. A choice of $\mu = \nu = m$ leads to the Haar measure on $\G \ba G$. The other two natural measures in this setup are then 
\begin{equation} \label{define_BR} m^{\BR} := m^{\mu^{\PS}, m}\end{equation} 
and
\begin{equation} \label{define_BR_*} m^{BR}_* := m^{m, \mu^{\PS}} \end{equation}
We refer to these as the unstable (respectively stable) Burger-Roblin measures on $\G \ba G$. They are in some sense the only interesting measures on $\G \ba G$ that are invariant and ergodic for the action of $N^+$ (or respectively $N^-$); this fact is due to Roblin \cite{Ro} and described in our current language by \cite[Theorem 2.5]{Hee}. Note that they are always infinite unless $\G$ is a lattice, in which case they coincide with Haar measure on $\G \ba G$.

 \subsection{Patterson-Sullivan densities on horospheres and doubling properties}
 We now return to the delayed proof of Lemma \ref{doubling_lemma}.  This will follow in a straightforward manner from two facts. 

Firstly the measure $\nu_i$ is absolutely continuous (with Radon-Nikodyn derivatives uniformly bounded both above and below) to the restriction of the Patterson-Sullivan measure $\mu^{\PS}_{z_iMN^+}$ on the strong unstable leaf through $z_i$ to $U_i$. This is clear using the facts that
\begin{itemize}
\item the inclusion of $\hat R^\tau$ into $\G \ba G/M$ sends the measure $d\nu \times dt$ of maximal entropy to the BMS measure,
\item the measures $\nu_i$ are obtained by projecting $\nu$ from $\hat R_i$ onto $\hat U_i$
\item the BMS measure can be described in terms of Patterson-Sullivan measures as in \eqref{BMS_and_PS},
\item The Patterson-Sullivan measures satisfy a doubling property. 
\end{itemize}

Secondly the following Lemma holds.
\begin{lemma}
There exists $C > 0$ such that the following holds: for any $i$ and any $x \in \hat U_i$ and any $1 > \e > 0$ we have
$$ \nu_i(B_\e(x) \geq C \mu^{\PS}_{z_iMN^+}(B_\e(x))$$
for $z_i$ the base point of the rectangle $R_i$.
\end{lemma}
In other words for any $i, x$, and $\e$ the intersection of $B_\e(x)$ with $\hat U_i$ is a reasonably large fraction of $B_\e(x)$ 
\begin{proof}
This will follow simply from the relationship between $\nu$ and Patterson-Sullivan measures. 

For any $k$ and $u\in \hat U_k$ we have a holonomy map $H_{u, k} $ taking $\hat U_k$ into the strong unstable manifold through $u$. Moreover these maps are all Lipschitz with uniform Lipschitz bound $\beta$ independent of $j$ and the choice of $u$.

We choose $\alpha > 0$ to be the size our our Markov partition (see subsection \ref{Markov_sections}) and choose $k$ minimal such that $e^{\tau^{(k)}(x)} \epsilon  > \beta \alpha$. We note in particular that $e^{\tau^{(k)}(x)} \epsilon   < \beta \alpha e^{\alpha}$. 

Choose $j$ such that $\Pp^k(x) \in \hat U_j$. We then have a section $v : \tilde U_j \rightarrow \tilde U_i$ for $\Pp^k$. In fact this section is given exactly as
$$ y \mapsto H_{u, j} (y)\cdot  a_{-\tau^{(k)}(y)}.$$
Now $H_{u, j} (\tilde U_j)$ has diameter at most $\beta \cdot \alpha$, so $Y := H_{u, j} (\tilde U_j) \cdot a_{-\tau^{(k)}(x)} \subset \hat U_i$ has diameter less than $\epsilon$. On the other hand $x \in Y$, so $Y \subset B_\e(x) \cap \hat U_i$. 
Now write $t$ for $\tau^{(k)}(x)$ and bound
\begin{eqnarray*}
\nu_i(B_\e(x)) & \gg&mu^{\PS}_{z_iMN^+}(Y)\\
&=&e^{-\delta t}(y \cdot \mu^{\PS}_{z_ia_t MN^+}(Y\cdot a _t)\\
&\gg& e^{-\delta t}(y \cdot \mu^{\PS}_{z_j MN^+}(\hat U_j)
\end{eqnarray*}
using that invariance properties of Patterson-Sullivan measures on horospheres and continuity of the Busseman function. We therefore have that $\nu_i(B_\e(x))$ is bounded below by some essentially universal constant (independent, particularly of $x, i, \epsilon$) times $e^{-\delta t}$.

On the other hand 
\begin{eqnarray*} \mu^{\PS}_{z_iMN^+}(B_\epsilon(x)) &=& e^{-\delta t} \mu^{\PS}_{z_ia_t MN^+}(B_\epsilon(x)a_t )\\
&\geq& e^{-\delta t} \mu^{\PS}_{z_ia_t MN^+}(B_{ \beta_\alpha e^{\alpha}} (xa_t) )\\
\end{eqnarray*}
Now  the measure
$$mu^{\PS}_{z_ia_t MN^+}(B_{ \beta_\alpha e^{\alpha}} (xa_t) )$$
is bounded independently of $i$ and $xa_t$, since that it is constrained to vary within the compact set $R$ given by the union of our partition elements. Our lemma follows.
\end{proof}
With the Lemma in hand our doubling estimate is now straightforward.
\begin{proof}[Proof of Lemma \ref{doubling_lemma}] Choose $i, x\in \hat U_i$ and $\epsilon > 0$. Then 
$$ \nu_i(B_\epsilon(x)) \geq c \mu^{\PS}_{z_iMN^+}(B_\e(x))$$
by the lemma. We now apply the doubling property for the Patterson-Sullivan measure to obtain 
$$ \nu_i(B_\epsilon(x)) \gg \tilde c \mu^{\PS}_{z_iMN^+}(B_{2\e}(x)) \gg \hat  c \nu_i(B_{2\e}(x)). $$
Since the implied constant does not depend on $i, x$ or $\epsilon$ we're done. 
\end{proof}

\subsection{Exponential mixing and decay of matrix coefficients. }
Set
$$\alpha(g,{\Lambda(\G)}):=\inf\{|s|: (g h_s)^+\in \Lambda(\G)\} +\inf \{|s|: (g n_s)^-\in \Lambda(\G)\} + 1$$
where $s \rightarrow h_s$ is the natural parametrization of $N^+$ by $\mathbb{R}^{n-1}$ and similarly for $n_s$ and $N^-$. 
 For any compact subset $\mathcal Q\subset G$, 
$$\alpha (\mathcal Q,\Lambda(\G)):=\sup_{g \in \mathcal{Q}} \alpha (g, \Lambda(\G))< \infty .$$

The following theorem implies that Theorem \ref{main2} can be deduced from Theorem \ref{main1}. 
Let $\pi: G\to \Gamma \ba G$ be the canonical projection.
\begin{thm}\label{de} Let $\mathcal Q\subset G$ be a compact subset. 
Suppose that there exist constants $r_0 \in \mathbb{N}, c_\G>0$ and $\eta_\G>0$ such that
for any $\Psi,  \Phi\in C^{r_0}(\G\ba G)$ supported on $\pi(\mathcal Q)$,
\begin{multline} \label{bms3}  \int_{\G \ba G} \Psi (ga_t) \Phi (g) dm^{\BMS} =
 \tfrac{m^{\BMS}(\Psi) \cdot m^{\BMS}(\Phi)}{m^{\BMS}(\G\ba G)}  +O( c_\G \cdot ||\Psi||_{C^{r_0}} ||\Phi||_{C^{r_0}}
 \cdot e^{-\eta_\G t})\end{multline}
where the implied constant  depends only on $\mathcal Q$.
Then for any $\Psi,  \Phi\in C^1(\G\ba G)$ supported on $\pi(\mathcal Q)$, as $t\to +\infty$,
\begin{multline}   e^{(n +1-\delta)t} \int_{\G \ba G} \Psi (ga_t) \Phi (g) dm^{\Haar}\\ =
 \tfrac{ m^{\BR}(\Psi) \cdot m^{\BR_*}(\Phi)}{m^{\BMS}(\G\ba G)} +O( c_\G \cdot ||\Psi||_{C^{r_0}} ||\Phi||_{C^{r_0}}
 \cdot e^{-\eta_\G' t}) \end{multline}
where $\eta_\G' > 0$ depends on $\eta_\G$ and  the implied constant  depends only on $\mathcal Q$ and $\alpha (\mathcal Q,\Lambda(\G))$.
\end{thm} 
 
 The rest of this section is devoted to the proof of this fact.
The proof involves effectivizing the original argument of Roblin \cite{Ro}, extended in \cite{Sch}, \cite{OS}, \cite{MO},
while making the dependence of the implied constant
 on the relevant functions precise.

 For $\e>0$ and a subset $S$ of $G$, $S_\e$ denotes the set
$\{s\in S: d(s, e)\le \e\}$. Let $$P:=N^- AM .$$
 Then the sets $B_\e:= P_\e N^+_\e$, $\e>0$ form a basis of neighborhoods of $e$ in $G$.

\begin{prop} \label{imp1} Fix $x \in \G \ba G$ and $y\in xP_{\e_0}$ and put $\phi:=\Phi|_{yN^+_{\e_0}}\in C^{r_0}(yN^+_{\e_0})$. 
Then for $t>1$,
$$\int_{yN^+_{\e_0}} \Psi(yna_t) \phi(yn) d\mu_{yN^+}^{\PS}(yn)
=\frac{\mu_{yN^+}^{\PS}(\phi)}{|m^{\BMS}|} m^{\BMS}(\Psi)+O(c_\G \|\Psi\|_{C^{r_0}}\|\phi\|_{C^{r_0}} e^{-\eta_1 t})$$
where $\eta_1 > 0$ depends on $\eta_\G$ and the implied constant depends only on $\mathcal Q$ and $\alpha(\mathcal Q,\Lambda(\G))$.
\end{prop}  
\begin{proof} 
Set $R_0:=\sup \lbrace \alpha(g,\Lambda(\G)) : g \in yN_{\e_0}^+ \rbrace  +2$. For $\epsilon \in (0, 1)$ we may choose a smooth positive function $q_\e$ supported on $yN^+_{\e_0} N^-_{R_0}A_\epsilon M_\epsilon$
 such that
 $$\int_{yhN^-}\int_{\mathbb{R}}\int_M q_\e(na_tm) dm dt d\tilde \mu^{PS}_{yhN^-}(n) = 1$$
for all $h \in N^+_{\e_0}$   and that $\|q_\e\|_{C^{r_0}}\ll \e^{-\ell}$
for some positive $\ell$.   Define a $C^{r_0}$-function $ \Phi^\dag$ supported on $yN^+_{\e_0} N^-_{R_0} A_\epsilon M_\epsilon $ as follows:
 $$ \Phi^\dag(y h p):= {\phi(yh ) q_{\e} (yhp)}.$$
We have $m^{\BMS}(\Phi^\dag)=\mu_{yN^+}^{\PS}(\phi) $.
Now by \eqref{BMS_and_PS} and the hypothesis of Theorem \ref{de}, we have
\begin{align*}
&\int_{yN^+_{\e_0}} \Psi(yna_t) \phi(yn) d\mu_{yN^+}^{\PS}(yn) =(1+O(\e + R_0e^{-t})) \la a_{t} \Psi, \Phi^\dag\ra_{m^{\BMS}}\\
&= (1+O(\e + R_0e^{-t}))\left( \frac{\mu_{yN^+}^{\PS}(\phi) }{|m^{\BMS}|} m^{\BMS}(\Psi)+O(c_\G \e^{-\ell}e^{-{\eta_\G} t})\right)\\
&= \frac{\mu_{yN^+}^{\PS}(\phi) }{|m^{\BMS}|} m^{\BMS}(\Psi)+O(\e + c_\G \e^{-\ell} e^{-{\eta_\G} t} + R_0e^{-t})
\end{align*}
where the implied constant depends only on the $C^{r_0}$-norms of $\Psi$ and $\phi$ and $\mathcal Q$.
By taking 
\begin{equation} \e=e^{-{\eta_\G} t /(\ell+{\eta_\G})} \label{choose_epsilon_carefully_for_t_trick} \end{equation}
and by setting $\eta_1:={{\eta_\G}  /(\ell+{\eta_\G})},$ we obtain
$$\int_{yN^+_{\e_0}} \Psi(yna_t) \phi(yn) d\mu_{yN^+}^{\PS}(yn) =
\frac{\mu_{yN^+}^{\PS}(\phi) }{|m^{\BMS}|} m^{\BMS}(\Psi) +O(c_\G R_0 e^{-\eta_1 t}).$$
Since $R_0$ is bounded above in terms of  $\alpha(\mathcal Q,\Lambda(\G))$,
 this proves the claim.
\end{proof} 
Equidistribution of the Patterson-Sullivan density on horospheres implies a weighted equidistribution result for the lebesgue density on horospheres.

\begin{prop} \label{imp2} Keeping the same notation as in Proposition \ref{imp1} and \eqref{define_BR} we have
$$e^{(n - 1-\delta)t}\int_{yN^+_{\e_0}} \Psi(yna_t) \phi(yn) dn
=\frac{\mu_{yN^+}^{\PS}(\phi)}{|m^{\BMS}|} m^{\BR}(\Psi)+O(c_\G \|\Psi\|_{C^{r_0}}\|\phi\|_{C^{r_0}} e^{-\eta_1 t/2})$$
where the implied constant depends only on $\mathcal Q$ and $\alpha(\mathcal Q,\Lambda(\G))$.
\end{prop}
\begin{proof} This is the transverse intersection argument. It is potentially confusing, so we shall try to be careful and precise.  We lift $\Psi$ to give a function $\psi$ on $G$ with compact support in $xN^-_{\epsilon_0}A_{\epsilon_0}MN^+_{\epsilon_0}$ which is contained in a single fundamental domain. For simplicity we shall assume that $\psi, \phi$ are non-negative. 

We're currently taking integrals with respect to the Lebesgue type measures and would much rather take them with respect to Patterson-Sullivan measures whose equidistribution we already understand. We can make this change at the price of a small error and changing the function $\psi$ somewhat. That's our next job. 

More precisely, choose $R'$ large enough that the forward projections 
\[ (xnamN^+_{R'})^+ \mbox{ meet the limit set $\Lambda(\G)$}\]
for all $nam$ in $N^-_{\epsilon_0}A_{\epsilon_0}M_{\epsilon_0}$, and choose $R = R' + 1$. Choose a smooth positive bump function $q$ on $N^+_{R}$ with the property that 
\[ g(xnam) := \int_{xnamN^+_{R}} q((xnam)^{-1}h) d\mu^{\PS}_{xnamN^+}(h) >0 \]
for all $nam$ in $N^-_{\epsilon_0}A_{\epsilon_0}M_{\epsilon_0}$.  This auxiliary function allows us to modify $\psi$ as follows: let
\[ k(xnam) := \int_{xnamN^+_{\epsilon_0}} \psi(xnamh) d\mu^{\Leb}_{xnamN^+}(h) \] 
and set $\psi_1$ to be the function with compact support on $xN^-_{\epsilon_0}A_{\epsilon_0}MN^+_{R}$ given by 
\[ \psi_1(xnamh) := \frac{ k(xnam)q(h) }{g(xnam)}.\]
we observe the following properties of $\psi_1$:
\begin{itemize}
\item $\int_{zN_R^+} \psi_1(h) d\mu^{\PS}_{zN_R^+}(h) = k(z) = \int_z{N_{\epsilon_0}^+} \psi(h) d\mu^{\Leb}_{zN_R^+}(h) $ for all appropriate $z$, and consequently;
\item $m^{\BR}(\psi) = m^{\BMS}(\psi_1)$.
\end{itemize}
We fix also one piece of notation: denote by $\phi_{\alpha, +}(g)  := \sup_{h \in N^+_\alpha} \phi(gh)$ and write $\phi_{\alpha, -}$ for the corresponding expression with an infimum. Now calculate:
\begin{eqnarray*} (*) &:=& \sum_\G e^{(n - 1 - \delta)t}\int_{yN^+}  \psi(\g na_t) \phi(n) d\mu^{\Leb}_{yN^+}(n)\\
&=& \sum_\G e^{- \delta t}\int_{yN^+a_t}  \psi(\g na_t) \phi(na_{-t}) d\mu^{\Leb}_{yN^+a_t}(n)\\
&=& \sum_\G  e^{- \delta t}\int_{xN^-_{\epsilon_0}A_{\epsilon_0}MN^+_{\epsilon_0}}  \psi( n)d( \gamma_* \phi(na_{-t}) \mu^{\Leb}_{yN^+a_t}(n) ).\\
\end{eqnarray*}
Clearly the summation need only be taken over those $\g$ for which the intersection $\g \in \G: \g y a_t N^+ \cap xP_{\epsilon_0}$ is non-empty. For every such $\g$ let $z_\g$ be the (unique) intersection point. We have 
\begin{eqnarray*}
(*) &=& \sum_{z_\g}  e^{- \delta t} \int_{xN^-_{\epsilon_0}A_{\epsilon_0}MN^+_{\epsilon_0}}  \psi( n) \phi(\gamma^{-1} na_{-t}) \mu^{\Leb}_{z_\g N^+}(n) )\\
&\leq &  \sum_{z_\g}  e^{- \delta t}\int_{xN^-_{\epsilon_0}A_{\epsilon_0}MN^+_{\epsilon_0}}  \psi( n) \phi_{\epsilon_0e^{-t}, +}(\gamma^{-1} z_\g a_{-t}) \mu^{\Leb}_{z_\g N^+}(n) )\\
&=& \sum_{z_\g}  e^{- \delta t}\int_{xN^-_{\epsilon_0}A_{\epsilon_0}MN^+_{R}}  \psi_1( n) \phi_{\epsilon_0e^{-t}, +}(\gamma^{-1} z_\g a_{-t}) \mu^{\PS}_{z_\g N^+}(n) )\\
&\leq& \sum_{z_\g}  e^{- \delta t} \int_{xN^-_{\epsilon_0}A_{\epsilon_0}MN^+_{R}}  \psi_1( n) \phi_{(\epsilon_0 + R)e^{-t}, +}(\gamma^{-1} n a_{-t}) \mu^{\PS}_{z_\g N^+}(n) )\\
&=&  \sum_\G \int_{yN^+}  \psi_1(\g na_t) \phi_{(\epsilon_0 + R)e^{-t}, +}(n) d\mu^{\PS}_{yN^+}(n)
\end{eqnarray*}
We now apply equidistribution for the Patterson-Sullivan measure on unstable leaves; let $\tilde \psi_1$ be the function on $\G \ba G$ obtained by summing $\psi_1$ over $\G$ orbits. Then, using the fact that
\begin{eqnarray*} (*) &\leq& \frac{\mu_{yN^+}^{\PS}(\phi_{(\epsilon_0 + R)e^{-t}, +})}{|m^{\BMS}|} m^{\BMS}(\tilde \psi_1)+O(c_\G \|\tilde \psi_1\|_{C^{r_0}}\|\phi_{(\epsilon_0 + R)e^{-t}, +}\|_{C^{r_0}} e^{-\eta_1 t/2})  \\
&\leq&  \frac{\mu_{yN^+}^{\PS}(\phi)}{|m^{\BMS}|} m^{\BR}( \psi)+O(c_\G \| \psi\|_{C^{r_0}}\|\phi\|_{C^{r_0}} e^{-\eta_1 t/2})  \end{eqnarray*}
by changing the implied constant. In fact, to be really precise we should replace $\phi_{(\epsilon_0 + R)e^{-t}, +}$ by an appropriate smooth function before applying equidistribution. We assert that this is not hard. 

\end{proof}

The decay of matrix coefficients is a straightforward consequence of propositions \ref{imp1} and \ref{imp2}. First note by a partition of unity argument that we may consider functions $\Phi$,  $\Psi$
supported on $xB_{\e_0/2}, yB_{\e_0/2}$. 
We note that $dm^{\Haar}(pn)=dp dn$ where $dp$ is a left Haar measure on $P$ (which is equivalent to the Lebesgue measure), and hence
 $$\int_{\G\ba G} \Psi(xa_t) \Phi(x) dm^{\Haar}(x)=\int_{xp\in zP_{\e_0}} \int_{xpN^+_{\e_0}}
 \Psi(xpna_t) \Phi(xp n) dn dp.$$

Hence applying Propositions \ref{imp1} and \ref{imp2} for each $y=xp\in xP_{\e_0}$,
we deduce that
 \begin{align*}&e^{ (n-1-\delta) t}  \int_{\G\ba G} \Psi(xa_t) \Phi(x) dm^{\Haar}(x)\\
 &=\int_{xp\in xP_{\e_0}} \left( \tfrac{m^{\BR}(\Psi)\mu^{\PS}_{xpN^+} (\Phi|_{xpN^+_{\e_0}})}{|m^{\BMS}|} 
+ O(c_\G \|\Psi\|_{C^{r_0}} \|\Phi|_{xpN^+_{\e_0}} \|_{C^{r_0}} e^{-\eta_1 t/2}) \right) dp
\\&= \tfrac{m^{\BR}(\Psi)m^{\BR_*}(\Phi)}{|m^{\BMS}|} + O(c_\G \|\Psi\|_{C^{r_0}} \|\Phi \|_{C^{r_0}} e^{-\eta_1 t/2}) 
\end{align*}
as required.

\end{document}